\documentclass[11pt]{article}
\usepackage{amsmath, amssymb, amsthm}
\usepackage{enumitem}                                 
\usepackage[margin=1.0in]{geometry}

\usepackage{color}

\AtEndDocument{
\bigskip{\footnotesize%
\textsc{Department of Mathematics, the Pennsylvania State University, University
Park, PA 16802, USA} \par 
  \textit{E-mail address} \texttt{ amir.algom@mail.huji.ac.il}
} 

\bigskip{\footnotesize%
\textsc{Department of Mathematical Sciences, P.O. Box 3000, 90014 University of Oulu, Finland} \par 
  \textit{E-mail address} \texttt{  meng.wu@oulu.fi
} 
}}

\newtheorem{theorem}{Theorem}[section]
\newtheorem{Remark} [theorem]{Remark}

\newtheorem{Counter-example}[theorem]{Counter example}
\newtheorem{Claim}[theorem]{Claim}

\newtheorem{Lemma}[theorem]{Lemma}
\newtheorem{Proposition}[theorem]{Proposition}

\newtheorem{Definition}[theorem]{Definition}

\newtheorem{Conjecture}[theorem]{Conjecture}

\newtheorem*{theorem*}{Theorem}

\DeclareMathOperator*{\Dim}{Dim}

\newcommand{\supp}{\text{supp}}
\newcommand{\diam}{\text{diam}}

\newcommand\blfootnote[1]{%
  \begingroup
  \renewcommand\thefootnote{}\footnote{#1}%
  \addtocounter{footnote}{-1}%
  \endgroup
}

\usepackage{amsthm}

\newtheorem{innercustomgeneric}{\customgenericname}
\providecommand{\customgenericname}{}
\newcommand{\newcustomtheorem}[2]{%
  \newenvironment{#1}[1]
  {%
   \renewcommand\customgenericname{#2}%
   \renewcommand\theinnercustomgeneric{##1}%
   \innercustomgeneric
  }
  {\endinnercustomgeneric}
}

\newcustomtheorem{customthm}{Theorem}
\newcustomtheorem{customcong}{Conjecture}

\title{Improved versions of some  Furstenberg type slicing Theorems for self-affine carpets}

\author{Amir Algom and Meng Wu\blfootnote{M.W. is supported by  the Academy of Finland, project grant No. 318217.}}
\date{}

\begin{document}
\maketitle
\abstract{ Let $F$ be a Bedford-McMullen carpet defined by independent integer exponents. We prove that for every  line $\ell \subseteq \mathbb{R}^2$ not parallel to the major axes,
$$ \dim_H (\ell \cap F) \leq \max \left\lbrace 0,\, \frac{\dim_H F}{\dim^* F} \cdot (\dim^* F-1) \right\rbrace$$
and
$$ \dim_P (\ell \cap F) \leq \max \left\lbrace 0,\, \frac{\dim_P F}{\dim^* F} \cdot (\dim^* F-1) \right\rbrace$$
where $\dim^*$ is Furstenberg's star dimension (maximal dimension of microsets). This improves the state of art results on Furstenberg type slicing Theorems for affine invariant carpets.}

\section{Introduction} 
\subsection{Background and main results} \label{introduction}
Let $n\geq 2$ be an integer  and consider the $n$-fold map of the unit interval $T_n:[0,1]\rightarrow [0,1)$ 
\begin{equation} \label{Eq Tn}
T_n (x)=n\cdot x \mod 1.
\end{equation}
We say that integers $m,n\geq 2$ are independent, and write $m\not \sim n$, if $\frac{\log m}{\log n} \notin \mathbb{Q}$. In the 1960's Furstenberg formulated several Conjectures aiming to capture the  idea that if $m\not \sim n$ then expansions in base $n$ and in base $m$ should have no common structure. In 1967,  Furstenberg \cite{furstenberg1967disjointness} proved a landmark result of this form: If a closed subset of the torus $\mathbb{T}:=\mathbb{R}/\mathbb{Z}$ is invariant under both $T_m$ and $T_n$ then, assuming $m\not \sim n$, it is either finite or the entire torus. The measure theoretic analogue of this result, known as the $\times 2, \times 3$ Conjecture, remains open to this day: if $\mu$ is a Borel probability measure on $\mathbb{T}$, invariant under $T_m$ and $T_n$, then it is a convex combination of the Lebesgue measure and a purely atomic measure.

Some of the aforementioned conjectures that Furstenberg proposed are more geometric in nature. One of them is known as the Slicing Conjecture:  For $(u,t)\in  \mathbb{R} \times \mathbb{R}$ let $\ell_{u,t}$ denote the planar line with slope $u$ that intersects the $y$-axis at $t$ (notice that we exclude from notation lines that are parallel to the $y$-axis). For a set $A\subseteq \mathbb{R}^d$ we denote its  Hausdorff dimension by $\dim_H A$.
\begin{Conjecture} \cite{furstenberg1970intersections} \label{Conj Furs}
Let $\emptyset \neq X_1,X_2 \subseteq [0,1]$ be closed sets that are invariant under $T_m$ and $T_n$, respectively. If $m\not \sim  n$ then for all $u\neq 0$ and $t\in \mathbb{R}$,
\begin{equation} \label{Eq Furs bound}
\dim_H \left(X_1 \times X_2 \right)  \cap \ell_{u,t} \leq \max \lbrace 0, \, \dim_H X_1 + \dim_H X_2 -1 \rbrace.
\end{equation}
\end{Conjecture}
This Conjecture is a geometric manifestation of the idea ``if $m\not \sim n$ then expansions in base $n$ and in base $m$  have no common structure" in the sense that a slice of dimension larger than expected can be seen as some shared structure between $X_1$ and $X_2$. To explain why the  term on the right hand side of  \eqref{Eq Furs bound} is the expected bound, we recall the classical Marstrand slicing Theorem: For any set $X\subseteq \mathbb{R}^2$ and any  fixed slope $u$,
\begin{equation} \label{Eq. Marstrand}
\dim_H X\cap \ell_{u,t} \leq \max \lbrace 0,\, \dim_H X -1 \rbrace \text{ for Lebesgue almost every } t, 
\end{equation}
and this fails for any smaller value on the right hand side of  \eqref{Eq. Marstrand}. It is well known that for sets $X_1$ and $X_2$ as in Conjecture \ref{Conj Furs} 
$$\dim_H X_1 \times X_2 = \dim_H X_1 + \dim_H X_2.$$
So,  what Furstenberg conjectured is that for  $X=X_1 \times X_2$ as in Conjecture \ref{Conj Furs}, Marstrand's Theorem  holds for \textit{all} lines $\ell_{u,t}$ such that $u\neq 0$, that is, lines not parallel to the major axes.

Some progress towards Conjecture \ref{Conj Furs} was made by Furstenberg himself in \cite{furstenberg1970intersections}, Wolff \cite{Wolff1999Kakeya}, and later by Feng, Huang and Rao \cite{feng2014affine}. In 2016 the Conjecture was  proved simultaneously and independently by  Shmerkin \cite{shmerkin2016furstenberg} and Wu \cite{wu2016proof}. In the case when $\dim_H X_1 + \dim_H X_2\le 1$, Yu \cite{Yu2021Wu} has simplified Wu's arguments and obtained some quantitative improvement to \eqref{Eq Furs bound}. Austin \cite{austin2020new} recently gave a new short proof of Conjecture \ref{Conj Furs}. 

The phenomenon predicted by Furstenberg was later shown to hold, in an appropriate sense, in a  class of sets  that strictly includes certain product sets as in Conjecture \ref{Conj Furs}, called Bedford-McMullen carpets. These carpets are defined as follows: let $m, n$ be integers greater than one.  Let 
\begin{equation*}
\emptyset \neq  D \subseteq \lbrace 0,...,m-1 \rbrace \times \lbrace 0,...,n-1 \rbrace 
\end{equation*}
and define
\begin{equation*}
F = \left\lbrace (\sum_{k=1} ^\infty \frac{x_k}{m^k}, \sum_{k=1} ^\infty \frac{y_k}{n^k}) :\quad  (x_k,y_k) \in D\right\rbrace.
\end{equation*}
The set $F$ is then called a Bedford-McMullen carpet with defining exponents $m,n$, and allowed digit set $D$.  They are named after Bedford \cite{bedford1984crinkly} and McMullen \cite{mcmullen1984hausdorff} who calculated their dimensions. 

To recall the latest  results about slicing Theorems for Bedford-McMullen carpets we need the notion\footnote{This is the same notion as Assouad dimension (see e.g. \cite{fraser2020book}), but for consistency with previous papers on the subject we work here with $*$-dimension.}  of star-dimension: For a set $A \subseteq [0,1]^d$ we define
\begin{equation*} 
\dim^* A:= \sup \lbrace \dim_H M: \, M  \text{ is a microset of } A \rbrace
\end{equation*}
where microsets of $A$ are limits in the Hausdorff metric on subsets of $[-1,1]^2$ of ``blow-up" of increasingly small balls about points in $A$ (see e.g. \cite[Section 2.2]{algom2018slicing} for more details). This notion was introduced and studied by Furstenberg in \cite{furstenberg2008ergodic}.  Mackay \cite{mackay2011assouad} gave a closed combinatorial formula for $\dim^* F$ for any Bedford-McMullen carpet $F$ in terms of $m,n$ and $D$. 

Returning to slicing theorems, Algom \cite[Theorem 1.2]{algom2018slicing} proved that for any Bedford-McMullen carpet  $F$ with independent exponents $m\not \sim n$
\begin{equation} \label{Algom slicing}
\dim^* F \cap \ell_{u,t} \leq \max \lbrace \dim^* F -1, 0 \rbrace,\quad \text{ for all } (u,t)\in \mathbb{R}^2 \text{ such that } u\neq 0.
\end{equation}
We remark that very recently  B\'{a}r\'{a}ny,   K\"{a}enm\"{a}ki, and Yu \cite{Barany2021finer}, obtained similar results about slices through some non-carpet planar self-affine sets. Now, it is known \cite[Chapter 4]{bishop2013fractal} that for any Bedford-McMullen carpet $F$, writing $\dim_B F$ for its  box dimension and $\dim_P F$ for its packing dimension, 
\begin{equation} \label{inequalties dim}
\dim_H F \leq \dim_P F =\dim_B F \leq \dim^* F
\end{equation}
and that these inequalities are strict unless $F$ is Ahlfors regular.  So, for Ahlfors regular carpets the results of \cite{algom2018slicing} are optimal for all notions of dimension previously discussed. However, in some sense ``most" Bedford-McMullen carpets are not Ahlfors regular  \cite[Chapter 4]{bishop2013fractal}. It is thus the main purpose of this paper to improve \eqref{Algom slicing} for both the Hausdorff and the packing dimension of slices in the non-Ahlfors regular setting, and to relate them to the corresponding dimensions of the underlying carpet. Here is our main result:

\begin{theorem} \label{Main Theorem}
Let $F$ be a Bedford-McMullen carpet with exponents $(m,n)$. If $m \not \sim n$ then  for all $u\neq 0$ and $t\in \mathbb{R}$,
\begin{enumerate}
\item $ \dim_H (\ell \cap F) \leq \max \left\lbrace 0, \, \frac{\dim_H F}{\dim^* F} \cdot (\dim^* F-1) \right\rbrace$.

\item $\dim_P (\ell \cap F) \leq \max \left\lbrace 0, \, \frac{\dim_P F}{\dim^* F} \cdot (\dim^* F-1) \right\rbrace$.
\end{enumerate}
\end{theorem}

Some remarks are in order: First, since for non Ahlfors regular carpets the inequalities in \eqref{inequalties dim} become strict,  Theorem \ref{Main Theorem} does indeed  improve \eqref{Algom slicing}. Secondly, it is a natural question (see e.g.  \cite[Question 8.3]{fraser2020fractal}) if Theorem \ref{Main Theorem} may be upgraded to a Marstrand-like result of the form
$$ \dim_H (\ell_{u,t} \cap F) \leq \max \lbrace 0,\, \dim_H F-1  \rbrace,\quad \text{ for all } (u,t)\in \mathbb{R}^2 \text{ such that } u\neq 0.$$
Our methods currently fall short of proving such a strong statement. This will be explained in the next Section, where we outline the proof of Theorem \ref{Main Theorem}.

\subsection{Sketch of the proof of Theorem \ref{Main Theorem} } \label{Section sketch}
 Fix a Bedford-McMullen carpet $F$ with digit set $D$ and exponents $m\not \sim n$. We always assume, without loss of generality, that $m>n$. This implies that $\theta := \frac{\log n}{\log m} \not \in \mathbb{Q}$ is in $(0,1)$.  Let $\ell_0 \subseteq \mathbb{R}^2$ be an affine line with slope $m^{u_0}$ where $u_0 \in [0,1)$, which  may be assumed without any loss of generality. We want to bound $\dim_H F\cap \ell_0$ - the bound for $\dim_P F\cap \ell_0$ is obtained in a similar manner.

For every $u\in \mathbb{T}:= \mathbb{R}/\mathbb{Z}$, we define a map $\Phi_u : [0,1]^2 \rightarrow [0,1]^2$ via
\begin{equation*} 
\Phi_u (x,y) = 
     \begin{cases}
      (T_{m} (x), \, T_{n}  (y))   &\quad\text{if } u\in [1 - \theta,1) \\
       ( x, \, T_{n} (y))  &\quad\text{if } u\in [0, 1-\theta). 
     \end{cases}
\end{equation*}
Let  $R_\theta : \mathbb{T} \rightarrow \mathbb{T}$ denote the translation by $\theta$ map
     $$ R_\theta(t)=t+\theta \mod 1.$$
For a measure $\mu$ on $[0,1]^2$, a point $z=(x,y) \in \supp (\mu)$, and $u\in \mathbb{T}$ we define a ``magnifying" map  via
$$ M(\mu,(x,y),u) = 
\begin{cases}
      (\mu^{\mathcal{D}_{m} (x)\times \mathcal{D}_n (y)}, \Phi_u (z), R_\theta (u) )   &\quad\text{if } u\in [1 - \theta,1) \\
       (\mu^{[0,1] \times \mathcal{D}_n (y)}, \Phi_u (z), R_\theta (u) )  &\quad\text{if } u\in [0, 1-\theta) \\
     
     \end{cases}$$
where   $\mathcal{D}_p (w)$ is the unique cell of the partition of $\mathbb{R}$
$$\mathcal{D}_p = \left\lbrace \left[\frac{i}{p}, \frac{i+1}{p}\right),\quad i\in \mathbb{Z} \right\rbrace$$
that contains $w$, and the measure $\mu^{\mathcal{D}_{m} (x) \times \mathcal{D}_n (y)}$  is the push-forward via $T_m \times T_n$ of the conditional measure of $\mu$ on $\mathcal{D}_{m} (x) \times \mathcal{D}_n (y)$. The measure $\mu^{[0,1]\times \mathcal{D}_n (y)}$ is defined similarly.

By Frostman's Lemma we may find a Borel probability measure $\mu_0$ supported on $\ell_0\cap F$ such that $\dim \mu_0 =\dim_H \ell_0\cap F-o(1)$ (see Section \ref{Section dimension} for a discussion on dimension theory for measures).  Roughly speaking, we pick a $\mu_0$ typical point $(x_0,y_0)$  and find a sequence $N_j$ such that:
$$\frac{1}{N_j} \sum_{k=0} ^{N_j-1} \delta_{M^k(\mu_0,(x_0,y_0),u_0)}$$
converges to an $M$-invariant distribution $Q$, such that for $Q$ a.e.-$\omega$ the measure $\mu_\omega$ is supported on a product set $X_\omega$ (which is a microset of $F$), and:
\begin{enumerate} [label=(\roman*)]
\item For every $\omega$ we have $\dim_H X_\omega \leq \dim^* F$. 

\item $\dim \mu_0 \leq  \int  \dim \mu_\omega \, dQ(\omega) $.

\item $\int \dim_H X_\omega \, dQ(\omega) \leq \dim_H F$.

\item Let $Q= \int Q_\xi d\tau(\xi)$ denote the ergodic decomposition of $Q$. For $\tau$-a.e. $\xi$, if $Q_\xi$ is supported on measures with strictly positive dimension, then for $Q_\xi$ a.e. $\omega$ we have
$$\dim \mu_\omega \leq \dim_H X_\omega - 1$$ 
\end{enumerate}
Property (i) is an easy consequence of the fact that $X_\omega$ is a microset of $F$. Property (ii) is  a general feature of CP distributions (see Section \ref{Section CP})  -  and $Q$ is such a distribution.  Properties (iii)-(iv) are the main innovations of this paper:  For Property (iii), we first note that it is not a trivial consequence of our construction, since in general $X_\omega$ is not a subset of $F$. Thus, for property (iii)  we rely, among other things, on a concise choice of  the sequence $N_j$ and general properties of entropy. Part (iv) relies on a geometric consequence of Sinai's factor Theorem proved by Wu \cite[Theorem 6.1]{wu2016proof}. However, new ideas are required  since Wu's original argument for deducing (iv) from this result as in \cite{wu2016proof} does not apply in our setting, as the measures arising from $M$-orbits are not all supported on the same set $X_\omega$. We also remark that in practice we will prove our required bounds by studying the entropy of certain measures on $X_\omega$ rather than considering $X_\omega$ itself. This can be seen as another reason why our approach gives more refined results than \eqref{Algom slicing}: In \cite{algom2018slicing} Algom worked directly with the microset $X_\omega$ that usually satisfies $\Pi_2 (X_\omega) = \Pi_2 (F)$  where $\Pi_2(x,y)=y$, which resulted with the loss of some information.

Once a distribution $Q$ satisfying the properties above has been produced, an elementary optimization argument yields the inequality as in Theorem \ref{Main Theorem} part (1).

Finally, we remark that if  for $Q$ a.e.-$\omega$ the measure $\mu_\omega$ has strictly positive dimension, then (i)-(iv) above would yield the strong Marstrand-type inequality
$$\dim_H F\cap \ell_0 \leq \max \lbrace 0, \dim_H F -1 \rbrace$$
We do not know, however, if it is possible to construct such a distribution.

\noindent{ \textbf{Organization}} In Section \ref{Section pre} we survey some tools we shall use from dimension theory, entropy, and the theory of CP distributions. We proceed to prove Theorem \ref{Main Theorem} part (2) in Section \ref{Section proof part (2)} and then, using a similar scheme, Theorem \ref{Main Theorem} part (1) in Section \ref{Section proof part (1)}.

\noindent{ \textbf{Notation}} We use the notation $o_\epsilon (1)$ to indicate a quantity going to $0$ as the positive $\epsilon\rightarrow 0$, and similarly $o_l (1)$ stands for  a quantity going to $0$ as the integer  $l\rightarrow \infty$.

\noindent{\textbf{Acknowledgements}} The authors  are grateful to Mike Hochman and Jonathan Fraser for their remarks on previous versions of this manuscript. 

\section{Preliminaries} \label{Section pre}

 \subsection{Hausdorff and packing dimensions of sets and measures} \label{Section dimension}
Recall that for a set $A$ in a compact metric space $X$, we denote its Hausdorff dimension by $\dim_H A$ and its packing dimension by $\dim_P A$. For an  exposition on these notions see Mattila's book \cite{mattila1999geometry}. Also, let  $\mathcal{P}(X)$ denote  the collection of all Borel probability measures on $X$.

Next, let $\mu \in \mathcal{P}(X)$. For every $x\in \supp(\mu)$  we define the lower pointwise dimension of $\mu$ at $x$ as
\begin{equation*}
\dim(\mu,x)=\liminf_{r\rightarrow 0} \frac{\log \mu (B(x,r))}{\log r}
\end{equation*}
where $B(x,r)$ denotes the closed ball of radius $r$ about $x$. We also define the upper pointwise dimension of $\mu$ at $x$ as
\begin{equation*}
\Dim(\mu,x)=\limsup_{r\rightarrow 0} \frac{\log \mu (B(x,r))}{\log r}.
\end{equation*}
The measure $\mu$ is called exact dimensional if the lower and upper pointwise dimensions of $\mu$ coincide and it is constant almost surely. In this case we denote this  quantity by $\dim \mu$.

Frostman's Lemma \cite[Chapter 10]{falconer1997techniques} allows one to find measures on a set $A$ that approximate its dimension:
\begin{equation*}
\dim_H A = \sup \lbrace s: \, \exists \mu \in \mathcal{P}(A) \text{ such that } \dim(\mu,x) \geq s \text{ almost surely } \rbrace;
\end{equation*}
\begin{equation*}
\dim_P A = \sup \lbrace s: \, \exists \mu \in \mathcal{P}(A) \text{ such that } \text{Dim}(\mu,x) \geq s \text{ almost surely } \rbrace.
\end{equation*}

Finally, let $A$ be a bounded set. For every $r>0$ let $\mathcal{N}_r (A)$ denote the minimal number of sets of diameter less that $r$ required to cover the set $A$. Then
$$\dim_B (A) = \lim_{r\rightarrow 0} \frac{\mathcal{N}_r (A)}{-\log r}$$
provided the limit exists. Otherwise, the upper box dimension $\overline{\dim}_B (A) $ is defined as the corresponding $\limsup$.

\subsection{Entropy, partitions, and approximate squares } \label{Section entropy}

Let $X$ be a compact metric space, $\mu \in \mathcal{P}(X)$, and let $\mathcal{A}$ denote a  finite measurable partition of  $X$. Recall that the Shannon entropy of $\mu$ with respect to $\mathcal{A}$ is defined as
\begin{equation*}
H(\mu,\mathcal{A}) = -\sum_{A\in \mathcal{A}} \mu(A)\cdot \log \mu (A)
\end{equation*}
with the convention  $0\log 0=0$. The following  facts about Shannon  entropy are standard:

\begin{Proposition} \label{Lemma entropy}
Let $\mu\in \mathcal{P}(X)$ and let $\mathcal{A}$ be a finite measurable partition.
\begin{enumerate}
\item General upper bound: $H(\mu, \mathcal{A}) \leq \log |\lbrace A\in \mathcal{A}:\, A\cap \supp(\mu) \neq \emptyset \rbrace|.$

\item Entropy is concave: Suppose we have a disintegration of $\mu$, given by $\mu = \int \mu_\omega dQ(\omega)$. Then
$$ H(\mu, \mathcal{A}) \geq \int H(\mu_\omega, \mathcal{A}) dQ(\omega).$$

\item The Gibbs inequality: Let $n\in \mathbb{N}$ and  let $(p_1,...,p_n)$ and $(q_1,...,q_n)$ be two probability vectors. Then
$$-\sum_{i=1} ^n p_i \log p_i \leq -\sum_{i=1} ^n p_i \log q_i$$
with equality if and only if $p_i=q_i$.
\end{enumerate}
\end{Proposition}

Next, let $m\geq 2$.  For every integer $p\geq 0$ let $\mathcal{D} ^d _{p}$ denote the $m^p$-adic partition of $\mathbb{R}^d$, that is,
\begin{equation*}
\mathcal{D} ^d _p = \left\lbrace \prod_{i=1} ^d \left[\frac{z_i}{m^p}, \frac{z_i+1}{m^p}\right) :\quad  (z_1,...,z_d)\in \mathbb{Z}^d \right\rbrace.
\end{equation*}
We shall omit the superscript $d$ from our notation when its value is clear from context.

Also, given integers $m>n$, let $\theta:= \frac{\log n}{\log m}$. Recall that $R_\theta: \mathbb{T} \rightarrow \mathbb{T}$ is the group rotation
\begin{equation} \label{Eq def rotation1}
R_\theta(u)=u+\theta \mod 1.
\end{equation}
For every $k\in \mathbb{N}$ and $u\in \mathbb{T}$ we define the integer
\begin{equation} \label{Eq return time}
 \mathcal{R}(k,u):= |\lbrace 0 \leq i \leq k: \quad R_\theta ^i (u) \in [1-\theta,1) \rbrace|.
\end{equation}
Now, for every $u \in \mathbb{T}$ we define a sequence of partitions of $[0,1]^2$,  called approximate squares, as follows: for every $k\in \mathbb{N}$ we let
$$ \mathcal{A}_k ^u := \mathcal{D}_{m^{\mathcal{R}(k,u)}} \times \mathcal{D}_{n^k}.$$
The following Lemma is standard:
\begin{Lemma} \label{Lemma  number of returns} \cite[Claim 4.2]{algom2018slicing} 
There exists some $C>1$ such that for every $u\in \mathbb{T}$ and every $z\in [0,1]^2$:
\begin{enumerate}
\item $[\theta\cdot k] - C \leq \mathcal{R}(k,u) \leq [\theta\cdot k]+C$.
\item For every $z$ 
$$C ^{-1} \frac{1}{n^k} \leq \diam \left( \mathcal{A}_k ^u (x) \right) \leq C \frac{1}{n^k}$$
where $\mathcal{A}_k ^u (z) $ is the atom of the partition $\mathcal{A}_k ^u$ that contains $z$.
\end{enumerate}
\end{Lemma}

 \subsection{Bedford-McMullen carpets} \label{Section BMC}
Recall the definition of a Bedford-McMullen carpet $F$ with defining exponents $m,n$ and allowed digit set $D$ from Section \ref{introduction}. We will always assume, without loss of generality, that $m>n$. We remark  that $F$ is a self-affine set generated by an IFS consisting of maps whose linear parts are diagonal matrices. Let $\Pi_2:\mathbb{R}^2 \rightarrow \mathbb{R}$ denote the projection to the second coordinate, that is, $\Pi_2 (x,y)=y$. For every $j\in \Pi_2 (F)$, let
$$ D_j = \lbrace 0\leq i \leq m-1:\quad (i,j)\in D \rbrace$$
and
\begin{equation} \label{Eq aj}
a(j):=  |D_j|.
\end{equation}
Recall that we have denoted
$$ \theta= \frac{\log n}{\log m} \in (0,1). $$
The following Theorem, due to Bedford and McMullen independently, describes the various dimensions of $F$.
\begin{theorem} \cite{bedford1984crinkly, mcmullen1984hausdorff} \label{Theorem BMC}
Let $F$ be a Bedford-McMullen carpet. Then:
\begin{enumerate}
\item $\dim_H F = \frac{\log \left( \sum_{j\in\Pi_2(D)} a(j)^\theta \right)}{\log n}$.

\item $\dim_B F=\dim_P F = \frac{\log | \Pi_2 (D)|}{\log n} + \frac{\log \frac{|D|}{|\Pi_2 (D)|}}{\log m}$. In particular, $\dim_B F$ exists as a limit.
\end{enumerate}
\end{theorem}

The following Lemma describes what happens as we zoom into  $F$ via the approximate squares $\mathcal{A}_k ^u$ defined in Section \ref{Section entropy}. For $(x,y) \in F$, write
$$(x,y) = \left( \sum_{k=1} ^\infty \frac{x_k}{m^k}, \sum_{k=1} ^\infty \frac{y_k}{m^k}\right),\quad \text{ where } (x_k,y_k)\in D$$
for the corresponding base $m$ and base $n$ expansions of $x$ and $y$, respectively, that bare witness to $(x,y)\in F$. Notice that these expansions may not be unique in general. In this case, if say $x$ has two such base $m$ expansions, we choose the one that ends with $0$'s (the lexicographically larger one). Let $\Pi_1(x,y)=x$ denote the  projection to the first coordinate.

\begin{Lemma} \cite[Section 7]{algom2016self} \label{Lemma 0.0}
Let $F$ be a Bedford-McMullen carpet and let $(x,y)\in F$. Writing
$$(x,y) = \left( \sum_{k=1} ^\infty \frac{x_k}{m^k}, \sum_{k=1} ^\infty \frac{y_k}{m^k}\right),\quad \text{ where } (x_k,y_k)\in D,$$
we have, for every $k\in \mathbb{N}$ and $u\in \mathbb{T}$, that the set 
$$T_m ^{\mathcal{R}(k,u)} \circ  \Pi_1 \left( \mathcal{A}_k ^u (x,y) \cap F \right) $$
is contained in
$$\left\lbrace \sum_{i=1} ^\infty \frac{b_i}{m^i}: b_i \in \lbrace 0,...,m-1 \rbrace \text{ and for } 1\leq i\leq \frac{(1-\theta)k}{\theta}-o_k(C),\quad b_i\in D_{y_{i+\mathcal{R}(k,u)}} \right\rbrace$$
where $C$ and $\mathcal{R}(k,u)$ are as in Lemma \ref{Lemma  number of returns}.
\end{Lemma}
Lemma \ref{Lemma 0.0} can also be recovered from the analysis of K\"{a}enm\"{a}ki,  Ojala,  and Rossi \cite{kaenmaki2016rigid}.

 \subsection{Dynamical systems} \label{Section Ber mea.}
 In this paper a measure preserving system is a quadruple $(X, \mathcal{B},T,\mu)$, where $X$ is a compact metric space, $\mathcal{B}$ is the Borel sigma algebra, and  $T:X\rightarrow X$ is a  measure preserving map: $T$ is Borel measurable and $T\mu = \mu$. Since we always work with the Borel sigma-algebra, we shall usually just write $(X,T,\mu)$. When the space $X$ is clear from context we shall sometimes just write $(T,\mu)$.  We also recall that a dynamical system is ergodic if and only if the only invariant sets are trivial. That is, if $B\in \mathcal{B}$ satisfies $T^{-1} (B) = B$ then $\mu(B)=0$ or $\mu(B)=1$. 
 
 A class of examples is given by symbolic dynamical systems: For $n\in \mathbb{N}$ at least $2$, let $X=\lbrace 0,...,n-1 \rbrace^\mathbb{N}$ and  $T=\sigma$  be the shift map $\sigma:[n]^\mathbb{N}\rightarrow [n]^\mathbb{N}$ defined by $\sigma(\omega) = \xi$ where $\xi (k) =\omega(k+1)$ for every $k$. We equip this space with the compatible compact metric $d$  defined  by
\begin{equation} \label{The metric on symbolic}
d (\omega,\xi) = \left( \frac{1}{n} \right) ^{\min\lbrace k : \,  \omega_k \neq \xi_k \rbrace}.
\end{equation} 
 
 We will have occasion to use  the ergodic decomposition Theorem: Let $(X,T,\mu)$ be a dynamical system. Then there is a  map $X\rightarrow \mathcal{P}(X)$, denoted by $\mu \mapsto \mu_{x}$, such that:
\begin{enumerate}
\item The map $x\mapsto \mu_x$ is measurable with respect to the sub $\sigma$-algebra $\mathcal{I}$ of $T$ invariant sets.

\item $\mu = \int \mu_x d\mu(x)$.

\item For $\mu$ almost every $x$, $\mu_x$ is $T$ invariant and ergodic. The measure $\mu_x$ is called the ergodic component of $x$.
\end{enumerate}

Recall that  if  $\mu\in \mathcal{P}(X)$ is a $T$ invariant measure we may define its metric entropy with respect to $T$, a quantity that we shall denote by $h(\mu,T)$. As there is an abundance of excellent texts on entropy theory (e.g. \cite{Walters1982ergodic}), we omit a discussion on entropy  here. We do recall that entropy is affine in the sense that if $\mu,\nu,\eta$ are $T$ invariant measures such that for some $p\in (0,1)$ we have $p\cdot \nu + (1-p)\cdot \eta = \mu$
then $$h(\mu, T) = p\cdot h(\nu,T)+(1-p)h(\eta,T).$$

Finally, we will consider dynamical systems of the form $([0,1], \mu, T_n)$ (recall \eqref{Eq Tn}). In this case we have the following useful result, which is an immediate consequence of the Shannon-McMillan-Breiman theorem and Billingsley's lemma:
\begin{theorem}  \label{Theorem 9.1}
Let $\mu\in \mathcal{P}([0,1])$ be a $T_n$ invariant and ergodic measure.  Then $\mu$ is exact dimensional and
\begin{equation*}
\dim \mu = \frac{h(\mu,T_n)}{\log n}.
\end{equation*} 
\end{theorem}

\subsection{CP distributions with respect to approximate squares} \label{Section CP}
The theory of CP distributions that we discuss in this section originated implicitly with Furstenberg in \cite{furstenberg1970intersections}. It was then reintroduced by Furstenberg in \cite{furstenberg2008ergodic}, and has since been used by many authors, notably by Hochman and Shmerkin in \cite{hochman2009local}. In particular, CP distributions played a crucial role  in both author's  works \cite{algom2018slicing,wu2016proof} about slicing Theorems. Here we will only discuss a special case of this machinery, using the approximate squares $\mathcal{A}_k ^u  (x) $ from Section \ref{Section entropy} as our partitions.  

As is standard in this context, for a compact metric space $X$   the elements of $\mathcal{P}(X)$ are called measures, and the elements of $\mathcal{P}(\mathcal{P}(X))$, measures on the space of measures, are called distributions.

Fix integers $m>n>1$ and recall that $\theta:=\frac{\log n}{\log m}$.  Recall the definition of the approximate squares $\mathcal{A}_k ^u$ from Section \ref{Section entropy}. For $\mu\in \mathcal{P}([0,1]^2)$, $u\in \mathbb{T}$, $k \in \mathbb{N}$,  and $x\in \supp(\mu)$, recalling \eqref{Eq return time}, let  
$$\mu^{\mathcal{A}_{k} ^u (x) } = \left( T_m ^{\mathcal{R}(k,u)} \times T_n ^k \right) \left( \mu_{\mathcal{A}_k ^u (x)} \right),\quad \text{where } \mu_{\mathcal{A}_k ^u (x)} (B) = \frac{\mu (\mathcal{A}_k ^u (x)\cap B)}{\mu(\mathcal{A}_k ^u (x))}.$$
That is, $\mu^{\mathcal{A}_{k} ^u (x)} $ is the push-forward of the conditional measure of $\mu$ on $\mathcal{A}_k ^u (x)$ via the map
$$T_m ^{\mathcal{R}(k,u)} \times T_n ^k.$$

Now, for every $u\in \mathbb{T}:= \mathbb{R}/\mathbb{Z}$, we define a map $\Phi_u : [0,1]^2 \rightarrow [0,1]^2$, by
\begin{equation} \label{Furstenberg skew}
\Phi_u (x,y) = 
     \begin{cases}
      (T_{m} (x), \, T_{n}  (y))   &\quad\text{if } u\in [1 - \theta,1) \\
       ( x, \, T_{n} (y))  &\quad\text{if } u\in [0, 1-\theta). \\
     
     \end{cases}
\end{equation}
We also define
 \begin{equation}\label{eq def Omega1}
 \Omega= \lbrace (\mu,z):\quad \mu\in \mathcal{P}([0,1]^2), \quad z\in \supp(\mu)\rbrace.
 \end{equation}
Finally, we define a ``magnification" map $M:\Omega \times \mathbb{T} \rightarrow \Omega \times \mathbb{T}$ via
$$ M(\mu,(x,y),u) = 
\begin{cases}
      (\mu^{\mathcal{D}_{m} (x)\times \mathcal{D}_n (y)}, \, \Phi_u (x,y), \, R_\theta (u) )   &\quad\text{if } u\in [1 - \theta,1) \\
       (\mu^{[0,1] \times \mathcal{D}_n (y)}, \, \Phi_u (x,y), \, R_\theta (u) )  &\quad\text{if } u\in [0, 1-\theta).
  \end{cases}$$
 Recall that $R_\theta$ is the rotation by $\theta$ map (see \eqref{Eq def rotation1}). 
Notice that for every $k\in \mathbb{N}$ and every $(\mu,z,u)\in \Omega\times \mathbb{T}$, the first coordinate of $M^k (\mu,z,u)$ is exactly $\mu^{\mathcal{A}_k ^u (z)}$. 
 
\begin{Definition} \label{Defition CP}
A  CP  distribution $Q$ with respect to the partition into approximate squares is a distribution $Q\in \mathcal{P}(\Omega \times \mathbb{T})$ such that:
\begin{enumerate}
\item $Q$ is $M$ invariant.

\item The  marginal distribution $Q_{1,2}$ of $Q$ on the first two coordinates $(\mu,x)$ of $\Omega$  is given by choosing first $\mu$ according to $Q_1$ (the marginal of $Q$ on $\mathcal{P}([0,1]^2)$) and then choosing $x$ according to $\mu$.
\end{enumerate} 
An ergodic CP distribution is a CP distribution  $Q$ that is $M$-ergodic.
 \end{Definition}
Note that a distribution  $Q\in \mathcal{P}(\Omega \times \mathbb{T})$ is a CP-distribution in the sense of the above definition if and only if its marginal $Q_{1,2}$ is a CP-distribution in the sense of  \cite{furstenberg2008ergodic,hochman2009local}.
 We shall sometimes abuse notation by referring to $Q_1$, the marginal of $Q$ on $\mathcal{P}([0,1]^2)$, as $Q$. In the following Theorem we group a few  facts about CP distributions. The first three were  proved by Furstenberg \cite{furstenberg2008ergodic}. The last one can be found in  \cite[Proposition 3.7, Lemma 7.3]{wu2016proof}. For a CP distribution $Q$  we define its dimension by
 \begin{equation*}
 \dim Q:=  \int  \dim \mu \, dQ(\mu,x,u).
 \end{equation*}
 \begin{theorem} \cite{furstenberg2008ergodic, wu2016proof} \label{Theorem CP properties} The following statements hold true:
 \begin{enumerate}
 \item The ergodic components of a CP distribution are, almost surely, themselves ergodic CP distributions.
 
 \item Let $Q$ be an ergodic CP distribution. Then $Q$ almost every measure $\mu$ is exact dimensional, and  $\dim \mu = \dim Q$.
 
 \item Let $Q$ be a CP distribution. Then $Q_2 = \int  \mu \, dQ(\mu,x,u)$, where $Q_2$ is the marginal of $Q$ on  $[0,1]^2$ (its second coordinate).
 
 \item Let $Q$ be an ergodic CP distribution. For every $\epsilon>0$ there exists some $r_0=r_0(\epsilon)>0$ such that:
 
  For every $r<r_0$,  $u\in \mathbb{T}$,  $k\in \mathbb{N}$, and $l\in \mathbb{N}$ large enough, we have for $Q$ almost every $(\mu,z,u)$
\begin{equation*} 
\inf_{y\in \mathbb{R}^2} H ( \mu^{\mathcal{A} ^u _k (z)} |_{\mathbb{R}^2 \setminus B(y,r)}, \mathcal{D}_{n^l}) \geq H ( \mu^{\mathcal{A} ^u _k (z)}, \mathcal{D}_{n^l}) -o_{\epsilon}(1){\cdot l\cdot\log n}= l\cdot \log n \cdot  ( \dim \mu -{o_{\epsilon} (1)}).
\end{equation*}
  \end{enumerate}
\end{theorem}
 
 Finally, we shall require a Theorem of Hochman and Shmerkin from \cite{hochman2009local}. Though it is not stated this way in \cite{hochman2009local}, it nonetheless follows directly from their local entropy averages machinery \cite[Section 4.2]{hochman2009local},  and their discussion in \cite[Section 7.5]{hochman2009local}:
 
 \begin{theorem} \cite{hochman2009local} \label{Theorem CP exist}
  Let $\mu \in \mathcal{P}([0,1]^2)$ be a measure such that for all $k\in \mathbb{N}$ and $u\in \mathbb{T}$,
  $$\mu( \partial  A) =0\, \text{ for every } A\in \mathcal{A}_k ^u.$$ 
  \begin{enumerate}
  \item[(1)]
  Suppose that $\Dim(\mu,x)\geq s$ for $\mu$-a.e. $x$. Then for $\mu$-a.e. $x$ and  every  $u\in \mathbb{T}$  there exists a subsequence $N_j$ such that
  \begin{equation*}
 \frac{1}{N_j} \sum_{k=0} ^{N_j-1} \delta_{M^k (\mu,x,u)} \rightarrow Q
 \end{equation*}
 where $Q$ is a CP distribution with $\dim Q \geq s$.
 
 \item[(2)]
  Suppose that $\dim(\mu,x)\geq s$ for $\mu$-a.e. $x$. Then for $\mu$-a.e. $x$, for every  $u\in \mathbb{T}$ and every subsequence $N_j$, there is a further subsequence $N_{j'}$ such that
  \begin{equation*}
 \frac{1}{N_{j'}} \sum_{k=0} ^{N_{j'}-1} \delta_{M^k (\mu,x,u)} \rightarrow Q
 \end{equation*}
 where $Q$ is a CP distribution with $\dim Q \geq s$.
  \end{enumerate}
\end{theorem}
\begin{Remark} \label{Remark - 0-remark}
\begin{itemize}
\item[(1)] Notice that the difference between (1) and (2) in Theorem \ref{Theorem CP exist} is that in part (1) we have to follow a specific subsequence to get $Q$, whereas in part (2) every subsequence will have a further subsequence that will yield such a distribution $Q$. 
\item[(2)] To see that  the distributions $Q$ as in Theorem \ref{Theorem CP exist} are $M$-invariant, note that their marginal on the $u$ coordinate must be the Lebesgue measure on $[0,1]$, so $M$ acts continuously on their support.
\end{itemize}
  \end{Remark}

\section{On the proof of Theorem \ref{Main Theorem} part (2)} \label{Section proof part (2)}
Let $F$ be a Bedford-McMullen carpet with exponents $m>n$ and digits $D$ such that $m\not \sim n$. Write $\theta:=\frac{\log n}{\log m}$. Let $\ell_{0}$ be a line not parallel to the major axes. Then the slope of $\ell_0$ can be written as $C\cdot m^{u_0}\neq 0$ for certain $u_0 \in [0,1)$ and some $C\neq 0$. We assume without loss of generality that $C=1$. 

\medskip
\medskip

\noindent{\textbf{From this point forward}} We work in  $\mathbb{T}$ and $\mathbb{T}^2$, so that the maps $T_m, T_n$  from \eqref{Furstenberg skew}  become continuous. This means that we think of $F$ and $\ell_0\cap F$ as subsets of $\mathbb{T}^2$ rather than $[0,1]^2$. Note that since $\mathbb{T}^2$ and $[0,1]^2$ are locally bi-Lipschitz equivalent,  the dimension of $\ell_0 \cap F$ as a subset of $\mathbb{T}^2$ is equal to its dimension as a subset of $[0,1]^2$.

\medskip

Let
$$ \gamma_0:= \dim_P \ell_0 \cap F$$
and let $\gamma < \gamma_0$. We will show that 
$$\gamma \leq \max \left\lbrace 0, \, \frac{\dim_P F}{\dim^* F} \cdot (\dim^* F-1) \right\rbrace .$$
It is clear that we may assume $\gamma>0$.

By Frostman's Lemma  we may find a probability measure $\mu_0 \in \mathcal{P}(\ell_0 \cap F)$ such that 
$$ \Dim(\mu_0,z)\geq \gamma ,\quad \text{ for } \mu_0 \text{ almost every } z.$$
In particular, $\mu_0$ is continuous (has no atoms). By Theorem \ref{Theorem CP exist} part (1) there  is a point $z_0 \in \ell_0 \cap F$ and a subsequence $N_j$ such that
\begin{equation} \label{Limit for Q}
 \frac{1}{N_j} \sum_{k=0} ^{N_j-1} \delta_{M^k (\mu_0,z_0,u_0)} \rightarrow Q
 \end{equation}
 where $Q$ is a CP distribution with 
\begin{equation} \label{dim Q}
\dim Q \geq \gamma.
\end{equation}
 
 Next, write 
 $$z_0=(x_0,y_0) = \left(\sum_{k=1} ^\infty \frac{x_k}{m^k}, \sum_{k=1} ^\infty \frac{y_k}{n^k}\right),\quad (x_k,y_k)\in D.$$
Notice that since $\mu_0$ is continuous, we may assume both $x_0,y_0 \notin \mathbb{Q}$, so that this representation is unique. Now, let
$$ \omega_0=(y_1, y_2,...) \in \left( \Pi_2 D \right) ^\mathbb{N} \subseteq \lbrace 0,...n-1 \rbrace^\mathbb{N}.$$
  Then, by perhaps moving to a further subsequence, we assume that there are $\sigma$ invariant measures $\nu,\eta, \rho \in \mathcal{P}((\Pi_2 D)^\mathbb{N})\subseteq \mathcal{P}(\lbrace 0,...n-1 \rbrace^\mathbb{N})$ such that:
 \begin{equation} \label{Eq for nu}
 \frac{1}{[N_j\cdot \theta]} \sum_{k=1} ^{[N_j \cdot \theta]} \delta_{\sigma ^k (\omega_0)} \rightarrow \nu,
 \end{equation}
  \begin{equation} \label{Eq for eta}
 \frac{1}{N_j - [N_j\cdot \theta]} \sum_{k=[N_j \cdot \theta]+1} ^{N_j} \delta_{\sigma ^k (\omega_0)} \rightarrow \eta,
 \end{equation}
  \begin{equation} \label{Eq for rho}
 \frac{1}{N_j} \sum_{k=1} ^{N_j} \delta_{\sigma ^k (\omega_0)} \rightarrow \rho.
 \end{equation}
Using \eqref{Eq for nu}, \eqref{Eq for eta} and \eqref{Eq for rho}, it is readily checked that  $\rho = \theta\cdot \nu + (1-\theta)\cdot \eta$.
 
The following Theorem is the key to the proof of Theorem \ref{Main Theorem} part (2). Recall the definition of $a(j)$ for $j\in \Pi_2 (D)$ from \eqref{Eq aj}.

\begin{theorem} \label{Key prop} Let $\lambda= Q( \lbrace \mu:\, \dim \mu >0 \rbrace)$. Then:
\begin{enumerate}
\item[(1)] $\gamma \leq \lambda\cdot (\dim^* F -1)$;

\item[(2)] 
$$\gamma+\lambda \leq \frac{\sum_{j\in \Pi_2 (D)} \nu([j])\log a(j)}{\log m} + \frac{h(\rho,\sigma)}{\log n},$$ where
  for  $j\in \lbrace 0,...,n-1\rbrace$ we write  $[j]=\lbrace \omega \in \lbrace 0,...,n-1\rbrace^\mathbb{N}:\quad \omega_1=j \rbrace$;

\item[(3)] $\gamma \leq \dim_P F -\lambda$.
\end{enumerate}
\end{theorem}


Theorem \ref{Key prop} implies Theorem \ref{Main Theorem} part (2): Indeed, combining parts (1) and (3) we get
$$\gamma \leq \min \lbrace \lambda\cdot (\dim^* F -1), \,  \dim_P F -\lambda\rbrace.$$
 An elementary optimization argument shows that for each  $0\le \lambda\le 1$, the right hand term of the above inequality is always bounded by the following quantity
 $$ \max \left\lbrace 0, \, \frac{\dim_P F}{\dim^* F} \cdot (\dim^* F-1) \right\rbrace.$$
 Hence we obtain the desired conclusion of Theorem \ref{Main Theorem} part (2).

We thus proceed to prove Theorem \ref{Key prop}: First, we will establish parts (1) and (2).  We will then show that part (3) follows from part (2).  

\subsection{On the proof of Theorem \ref{Key prop}} \label{Section proof of key prop}
\subsubsection{Preliminaries}
First, we extend the definition of the map $M$ from Definition \ref{Defition CP}: For every $u\in \mathbb{T}$ define a map $\sigma_u :\lbrace 0,...,n-1\rbrace^\mathbb{N} \rightarrow \lbrace 0,...,n-1\rbrace^\mathbb{N}$ via
 \begin{equation*}
\sigma_u (\omega) = \begin{cases}
      \sigma (\omega)   &\quad\text{if } u\in [1 - \theta,1) \\
       \omega  &\quad\text{if } u\in [0, 1-\theta). \\
     
     \end{cases}
 \end{equation*}
Recall \eqref{eq def Omega1} for the definition of the space $\Omega$. 
We define a new map  
$$T:\Omega \times   \mathbb{T} \times \lbrace 0,...,n-1\rbrace^\mathbb{N} \rightarrow \Omega \times   \mathbb{T} \times \lbrace 0,...,n-1\rbrace^\mathbb{N} $$ 
via
$$ T(\mu,z,u, \omega) = \left( M(\mu,z, u), \sigma_u (\omega) \right). $$
Recall that the sequence $\{N_j\}$ was chosen such that \eqref{Eq for nu}, \eqref{Eq for eta} and \eqref{Eq for rho} hold. By perhaps moving to a further subsequence of $\{N_j\}$ and using the irrationality of $\theta$, we may assume that
\begin{equation} \label{Eq equi for R}
\frac{1}{N_j} \sum_{k=0} ^{N_j-1} \delta_{T^k (\mu_0,z_0,u_0, \omega_0)} \rightarrow R, \text{ and } R \text{ is } T \text{ invariant.}
\end{equation} 
To see why we may assume  $R$ is $T$ invariant, we recall Remark \ref{Remark - 0-remark} part (2).
Notice that by \eqref{Limit for Q}, we have $R_{1,2,3} = Q$, where we recall that $R_{1,2,3}$ denotes the marginal of $R$ on the first $3$ coordinates. Also, recall that   for every $k \in \mathbb{N}$ and $u\in \mathbb{T}$ we write
$$ \mathcal{R}(k,u):= |\lbrace 0 \leq i \leq k: \quad R_\theta ^i (u) \in [1-\theta,1) \rbrace|$$
and by Lemma \ref{Lemma  number of returns} there is some uniform constant $C>0$ such that 
\begin{equation} \label{Eq number of returns}
[\theta\cdot k] - C \leq \mathcal{R}(k,u_0) \leq [\theta\cdot k]+C, \quad \forall k\in \mathbb{N}, u\in \mathbb{T}.
\end{equation}

Next, recall that for $j\in \Pi_2 (D)$ we defined
$$ D_j = \lbrace i: (i,j)\in D \rbrace.$$
For $\omega \in (\Pi_2 D) ^\mathbb{N}$  we denote
\begin{equation} \label{Eq for A omega}
A(\omega) = \left\lbrace \sum_{k=1} ^\infty \frac{b_k}{m^k}: b_k \in D_{\omega_k} \right\rbrace.
\end{equation}

The following Lemma gives a description of  $R$ typical points.
\begin{Lemma} \label{Lemma 0.1}
For $R$ almost every $(\mu, z, u, \omega)$ we have:
\begin{enumerate}
\item The measure $\mu$ is supported on a line with slope $m^u$.

\item $ \Pi_1 \left( \supp(\mu) \right) \subseteq A(\omega).$
\end{enumerate}
\end{Lemma}
\begin{proof}
Fix $(\mu, z, u, \omega)\in \supp (R)$. By \eqref{Eq equi for R}, there exists a sequence $k_p$ such that 
$$T^{k_p} (\mu_0,z_0,u_0, \omega_0) \rightarrow (\mu,z,u,\omega)$$
By the definition of $T$ the first coordinate of $T^{k_p}   (\mu_0,z_0,u_0, \omega_0)$ is $\mu_0 ^{\mathcal{A}_{k_p} ^{u_0} (z_0)}$, and the fourth coordinate is $\sigma^{\mathcal{R}(u_0,k_p)} (\omega_0)$. So, applying Lemma \ref{Lemma 0.0} we have that 
$$  \Pi_1 \left( \supp \left( \mu_0 ^{\mathcal{A}_{k_p} ^{u_0} (z_0)} \right) \right)$$
 is contained in
$$ \left\lbrace \sum_{i=1} ^{\infty} \frac{b_i}{m^i}: b_i \in \Pi_2 D, \quad \text {and for } 1\leq i \leq \frac{(1-\theta)k_p}{\theta}-o_{k_p}(C),\quad b_i \in D_{ \sigma^{\mathcal{R}(k,u_0)} (\omega_0) (i)} \right\rbrace$$
where by $\sigma^{\mathcal{R}(k,u_0)} (\omega_0) (i)$ we mean the $i$-th coordinate of $\sigma^{\mathcal{R}(k,u_0)} (\omega_0)$, and $C$ is the constant from Lemma \ref{Lemma  number of returns}. Since $\sigma^{\mathcal{R}(k_p,u_0)} (\omega_0) \rightarrow \omega$,  taking $p\rightarrow \infty$ yields part (2) of  the Lemma.

Part (1) is a consequence of the fact that for every $p \in \mathbb{N}$ the measure $\mu_0 ^{\mathcal{A}_{k_p} ^{u_0} (z_0)}$ is supported on a line with slope $m^{R_\theta ^{k_p} (u_0)}$, and since  $R_\theta ^{k_p} (u_0) \rightarrow u$.
\end{proof}

We also have the following estimate. Recall that for  $r>0$, $\mathcal{N}_r (A)$ denotes the minimal number of sets of diameter $\leq r$ required to cover the bounded set $A$.
\begin{Lemma} \label{Lemma 0.2}
Let $q\in \mathbb{N}$ be large. Then
$$ \int \frac{\log \mathcal{N}_{m^{-q}} \left( A(\omega) \right) }{q\log m} d R(\mu,z,u,\omega) = \frac{\sum_{j\in \Pi_2 (D)} \nu([j])\log a(j)}{\log m} +o_q(1).$$
\end{Lemma}
\begin{proof}
First, notice that by definition of $A_\omega$ (recall \eqref{Eq for A omega}), and since $a(j) = |D_j|$ for all $j\in \Pi_2 (D)$,
$$\prod_{k=1} ^q a(\omega_k) \leq  \mathcal{N}_{m^{-q}} \left( A(\omega) \right) \leq  \prod_{k=1} ^q a(\omega_k)\cdot 5$$
where the $5$ factor arises from the possible presence of elements with multiple base $m$ representation in $A(\omega)$. Therefore,
$$   \frac{\log \mathcal{N}_{m^{-q}} \left( A(\omega) \right)}{q\log m} = \frac{\sum_{k=1} ^q \log a(\omega_k)}{q \log m} +o_q(1).$$
Then by \eqref{Eq equi for R} and the previous equation it suffices to show that
\begin{equation} \label{Eq suffices}
 \lim_{j\rightarrow \infty} \frac{1}{N_j} \sum_{k=1} ^{N_j} \frac{\sum_{i=1} ^q \log a(\omega_{\mathcal{R}(k,u_0)+i} )}{q \log m} = \frac{\sum_{j\in \Pi_2 (D)} \nu([j]) \log a(j)}{\log m} +o_q(1).
\end{equation}
To this end, we first notice that by the definition \eqref{Eq for nu} of $\nu$, we have 
\begin{equation} \label{Eq 7}
\lim_{j\rightarrow \infty} \frac{1}{[\theta\cdot N_j]} \sum_{k=1} ^{[N_j \cdot \theta]} \frac{\log a(\omega_k)}{\log m} =\frac{\sum_{j\in \Pi_2 (D)} \nu([j])\log a(j)}{\log m}.
\end{equation}
Also, assuming $q \in \mathbb{N}$ is large and $p >q$ we have, by \eqref{Eq number of returns},
$$ \left| \lbrace k\in \mathbb{N}: \, \mathcal{R}(k,u_0)+1 \leq p \leq \mathcal{R}(k,u_0)+q \rbrace \right| = \frac{q}{\theta}(1+o_q (C))$$
and consequently,
$$ \sum_{k=1} ^{N_j} \sum_{i=1} ^q \log a(\omega_{\mathcal{R}(k,u_0)+i} ) = \frac{q}{\theta}(1+o_q (C)) \sum_{k=1} ^{[N_j \cdot \theta]} \log a(\omega_k) + o_{N_j} (1).$$
Dividing the latter equation by $N_j\cdot q\cdot \log m$ and taking $j\rightarrow \infty$, we see via \eqref{Eq 7} that  \eqref{Eq suffices} holds true. This implies the Lemma.
\end{proof}

Finally, let $\Xi: \lbrace 0,...,n-1 \rbrace^\mathbb{N} \rightarrow \mathbb{T}$ be the base $n$ coding map 
\begin{equation*}
\Xi(\omega) = \sum_{k=1} ^\infty \frac{\omega_k}{n^k}.
\end{equation*}
Recall the definition of the measure $\rho$ from \eqref{Eq for rho}. 
\begin{Lemma} \label{Lema rho is intesity}
The measure $\Xi(\rho)$ is $T_n$ invariant and satisfies 
$$ \Xi(\rho) = \int \Pi_2 (\mu)\, dR(\mu,z,t,\omega),$$
here $\Pi_2 (x,y)=y$ is the coordinate projection in $\mathbb{T}^2$.
\end{Lemma}
\begin{proof}
Recall from \eqref{Eq for rho} that
 \begin{equation*}
 \frac{1}{N_j} \sum_{k=1} ^{N_j} \delta_{\sigma ^k (\omega_0)} \rightarrow \rho.
 \end{equation*}
Also, $\Xi$ is a factor map in the sense that $\Xi \circ \sigma = T_n \circ \Xi$. So, since $\Xi$ is a continuous factor map, applying it to both sides of this equation yields
 \begin{equation*}
 \frac{1}{N_j} \sum_{k=1} ^{N_j} \delta_{T_n ^k (y_0)} \rightarrow \Xi(\rho).
 \end{equation*}
We also have
 \begin{equation*} 
\frac{1}{N_j} \sum_{k=0} ^{N_j-1} \delta_{T^k (\mu_0,z_0,u_0, \omega_0)} \rightarrow R.
\end{equation*} 
Combining the two last displayed equations, we see that $\Xi(\rho)$ equals the marginal of $R$ on the second coordinate $y$ of its projection to $\mathbb{T}^2$ with coordinates $(x,y)$. Now, by  Theorem \ref{Theorem CP properties}, $(x,y)$ is distributed according to 
$$ \int \mu \, dR(\mu,z,t,\omega).$$
So, the marginal of $R$ on the $y$ coordinate is given by
$$\int \Pi_2 (\mu) dR(\mu,z,t,\omega).$$
This proves the Lemma.
\end{proof}
\subsubsection{The skew product $S$}
For any $T$ invariant distribution $R'$ we denote
$$ \dim R '= \int \dim \mu \,  dR'(\mu,z,u,\omega)$$
which is equal to $\dim Q$ for our distribution $R$. Now, consider the ergodic decomposition of $R$,
$$R =\int R_\xi \, d\tau(\xi).$$
By Theorem \ref{Theorem CP properties}, almost every $R_\xi$ satisfies that its marginal on the first three coordinates $(R_\xi)_{1,2,3}$ is a CP distribution in the sense of Definition \ref{Defition CP}.  
$$ $$
\textbf{From this point forward}
\begin{equation} \label{Assumption positive dim}
\text{ Fix an ergodic component } R_\xi \text{ such that } \dim R_\xi >0.
\end{equation}  
Then for an $R_\xi$ typical $(\mu,z,u,\omega)$ we have by ergodicity 
\begin{equation} \label{Eq 8}
 \frac{1}{N}\sum_{k=1} ^N \mu^{A_k ^u (z)} \rightarrow \int \nu \, dR_\xi (\nu,z,u,\omega).
\end{equation}
Also, by  Lemma \ref{Lemma 0.1} for every $k$ we have
\begin{equation} \label{Eq 9}
\Pi_1 (\supp \left( \mu^{A_k ^u (z)} \right)  )\subseteq A(\sigma_{R_\theta ^k (u)}\circ ... \circ \sigma_{R_\theta (u)} \circ \sigma_u (\omega) ).
\end{equation}
Now, consider the measure $\kappa \in \mathcal{P}(\mathbb{T}^2 \times \mathbb{T} \times (\Pi_2 D)^\mathbb{N} )$ defined by
$$ \kappa = \int \mu \times \delta_{ \lbrace (u, \omega) \rbrace }   dR_\xi (\mu, z, u, \omega)$$
and let $S: \mathbb{T}^2 \times \mathbb{T} \times (\Pi_2 D)^\mathbb{N} \rightarrow \mathbb{T}^2 \times \mathbb{T} \times (\Pi_2 D)^\mathbb{N}$ be the map
$$ S(z,u,\omega) = (\Phi_u (z), \, R_\theta (u), \, \sigma_u (\omega))$$
where we recall that $\Phi_u$ was defined in \eqref{Furstenberg skew}.
\begin{Lemma} \label{Lemma apply Sinai to S}
The measure $\kappa$ is $S$ invariant and ergodic. Moreover, for $\kappa$ almost every $(z,u,\omega)$ there is an $R_\xi$ typical measure $\mu$ such that \eqref{Eq 8} and \eqref{Eq 9} hold true. 
\end{Lemma} 
\begin{proof}
Recall that $\Pi_{2,3,4}: \mathcal{P}(\mathbb{T}^2) \times \mathbb{T}^2 \times \mathbb{T} \times \lbrace 0,...,n-1\rbrace^\mathbb{N} \rightarrow  \mathbb{T}^2 \times \mathbb{T} \times \lbrace 0,...,n-1\rbrace^\mathbb{N}$ is the projection
$$(\mu,z,u,\omega) \mapsto (z,u,\omega).$$
Then $\Pi_{2,3,4} \circ T = S\circ \Pi_{2,3,4}$. By Theorem \ref{Theorem CP properties} part (3) we have that $\kappa = \Pi_{2,3,4} R_\xi$. In particular, $\kappa$ is  $S$ invariant. Moreover, $(S,\kappa)$ is a factor of the ergodic system $(T, R_\xi)$, and therefore it is ergodic.

The last assertion is an immediate consequence of the definition of $\kappa$, and since \eqref{Eq 8} and \eqref{Eq 9} are $R_\xi$ generic properties.
\end{proof}

Let us now introduce a generator for the system $(S,\kappa)$. We first recall the definition of generators: Let $(X,U)$ be a  dynamical  system, and  let $\mathcal{D}$ be a finite partition of $X$. Let $\mathcal{D}_k = \bigvee_{i=0} ^{k-1} U^{-i} \mathcal{D}$ denote the coarsest common refinement of $\mathcal{D}, U^{-1} \mathcal{D},...,U^{-k+1}\mathcal{D}$. The sequence $\mathcal{D}_k$ is called the filtration generated by $\mathcal{D}$ with respect to $U$.  Now, if the smallest sigma algebra that contains $\mathcal{D}_k$ for all $k$ is the Borel sigma algebra, we say that $\mathcal{D}$ is an $S$-generating partition  for $(X,U)$.

Back to our system $(S,\kappa)$, let 
$$\mathcal{C} = \left( \mathcal{D}_m \times \mathcal{D}_n \right) \times \lbrace [0,1-\theta), [1-\theta,1)\rbrace \times \lbrace [j]:j\in \Pi_2 D\rbrace$$
be a  partition of the space
$$ \mathbb{T}^2 \times \mathbb{T} \times (\Pi_2 D)^\mathbb{N}.$$
Write $\mathcal{W}= \lbrace [0,1-\theta), [1-\theta,1)\rbrace$.
\begin{Lemma} \label{Lemma generator}
The partition $\mathcal{C}$ is an $S$-generating partition. Moreover,  $\kappa (\partial C)=0$ for every $k\in \mathbb{N}$ and every $C\in \mathcal{C}_k$.
\end{Lemma}
\begin{proof}
The first assertion is an easy consequence of the fact that, as $k$ grows to infinity, the maximal diameter (working, say, with the $\sup$ metric) of an element in the partition $\mathcal{C}_k$ converges to $0$. For the second part, let $k\in \mathbb{N}$ and fix and element in $\mathcal{C}_k$. This element is of the form  $A\times W \times I$  where $W\in \mathcal{W}_k$, and for some  $u\in W$ we have that $I$ is a cylinder set in $(\Pi_2 D)^\mathbb{N}$ of length  $\mathcal{R}(k,u) \approx [k\cdot \theta]$,   and $A\in \mathcal{A} ^{u} _k$ (note that the latter sets are independent of the choice of $u\in W$). Notice that $\partial (I) = \emptyset$. Therefore, by two application of  the "product rule" for the boundary of product sets
\begin{equation*}
\partial (A\times W \times I) \subseteq \partial (A\times W) \times (\Pi_2 D) ^\mathbb{N}  \subseteq (\partial A \times W \times (\Pi_2 D) ^\mathbb{N} ) \bigcup (A \times \partial W \times (\Pi_2 D) ^\mathbb{N} ).
\end{equation*} 
Thus, 
\begin{equation} \label{Eq summonds}
\kappa(\partial (A\times W \times I))\leq \kappa (\partial A \times W \times (\Pi_2 D) ^\mathbb{N} ) + \kappa (A \times \partial W \times (\Pi_2 D) ^\mathbb{N} ).
\end{equation}

Now, the first summoned on the right hand side of equation \eqref{Eq summonds} is $0$. This is because $R_\xi$ typical $\mu$ has  positive  dimension by our choice of $R_\xi$ and  Theorem \ref{Theorem CP properties}. In particular, they are not atomic.  Also,  $R_\xi$ almost every $(\mu,z,u,\omega)$ satisfies that $\mu$ is supported on a line with slope $m^{u}$ by Lemma \ref{Lemma 0.1}. On the other hand, $\partial A$ is a union of four lines that are parallel to the major axes. To sum up, $R_\xi$ almost every $\mu$ is continuous and $\supp(\mu)$ intersects $\partial A$  in at most $2$ points, so $\mu(\partial A)=0$. Thus, the result follows from the definition of $\kappa$.

The second summoned is trivially $0$ since the marginal on the second coordinate of $\kappa$ is the Lebesgue measure $\mathcal{L}$, as this is the unique $R_\theta$ invariant measure, and $\partial W$ consists of two points.
\end{proof}

 \subsubsection{A geometric consequence of Sinai's factor Theorem}
We say that a sequence $\lbrace x_k \rbrace_{k\in \mathbb{N}} \subset \mathbb{T}$ is uniformly distributed (UD) if for every sub-interval $J\subseteq \mathbb{T}$ we have
\begin{equation*}
\frac{1}{N} | \lbrace 0\ \leq k \leq N-1: \, x_k \in J \rbrace| \rightarrow \mathcal{L}  (J), \quad \text{ where } \mathcal{L} \text{ is the Lebesgue measure on } \mathbb{T}.
\end{equation*}

In \cite{wu2016proof}, Wu proved following  result by appealing to the Sinai factor Theorem:
\begin{theorem} \label{Theorem 6.1} \cite[Theorem 6.1]{wu2016proof}
Let $(X,T,\mu)$ be an ergodic measure preserving system. Let $\mathcal{A}$ be a generator with finite cardinality, and let $\lbrace \mathcal{A}_k \rbrace_k$ denote the filtration generated by $\mathcal{A}$ and $T$. Suppose that $\mu (\partial A)=0$ for every $k\in \mathbb{N}$ and every $A\in \mathcal{A}_k$. Let $\beta \notin \mathbb{Q}$.

Then for any $\epsilon>0$ and for all ${l\ge l(\epsilon)}$ large enough  there exists a disjoint family of measurable sets $\lbrace C_i \rbrace_{i=1} ^{N(l,\epsilon)}, C_i \subset X,$ such that:
\begin{enumerate}
\item $\mu( \bigcup C_i) > 1-\epsilon.$

\item For every $1 \leq i \leq N(l,\epsilon)$, $| \lbrace A\in \mathcal{A}_l: C_i\cap A \rbrace| \leq e^{l\cdot \epsilon}$.

\item There exists another  disjoint family of measurable sets $\lbrace \tilde{C}_i \rbrace_{i=1} ^{N(l,\epsilon)}, \tilde{C}_i \subset X$, such that for every $1\leq i \leq N(l,\epsilon)$ we have:
\begin{itemize}
\item $C_i \subseteq \tilde{C}_i,$

\item $\mu (C_i) \geq (1-\epsilon) \mu (\tilde{C}_i)$,

\item for $\mu$ a.e. $x$ we have that the sequence 
\begin{equation*}
\lbrace R_\beta ^k (0) \in \mathbb{T} : k\in \mathbb{N} \text{ and }T^k (x)\in \tilde{C}_i \rbrace
\end{equation*}
is UD.
\end{itemize}
\end{enumerate}
\end{theorem}

\subsubsection{Three key estimates}
We begin by establishing two bounds via Theorem \ref{Theorem 6.1}. Recall the definition of the sets $A(\omega)$ as in \eqref{Eq for A omega}. Fix $\epsilon>0$, and note that in the construction below we use the same $\epsilon$ for all ergodic components $R_\xi$ with positive dimension. The parameter $l$ below  will  depend on both $\xi$ and $\epsilon$, with the dependence on $\xi$ being measurable.
\begin{Proposition} \label{Proposition two bounds}
Fix a $\kappa$  typical $(z,u,\omega)$ and  a corresponding $R_\xi$ typical measure $\mu$ satisfying \eqref{Eq 8} and \eqref{Eq 9}. Then, for our small $\epsilon>0$  and {all large $l\ge l(\epsilon,\xi)$}, there exists a set $\mathcal{N}=\mathcal{N}_\xi \subseteq \mathbb{N}$ such that
\begin{equation} \label{First bound}
 \mathcal{N}_{n^{-l}} \left( \bigcup_{k\in \mathcal{N}} \supp( \mu^{\mathcal{A}_k ^u (z)})  \right)\geq n^{l\cdot (\dim \mu+1 -{o_{\epsilon}(1)})}
\end{equation}
and for some uniform constant $C_1$, for  any $k'\in \mathcal{N}$ 
\begin{equation} \label{Second bound}
 \mathcal{N}_{n^{-l}} \left( \Pi_1 \left( \bigcup_{k\in \mathcal{N}} \supp( \mu^{\mathcal{A}_k ^u (z)}) \right)  \right)\leq C_1 \cdot  \mathcal{N}_{n^{-l}} \left( A \left(  \Pi_3 \circ S^{k'} \left(z,u,\omega \right) \right)   \right).
\end{equation}
\end{Proposition}
We remark that $\Pi_3 \circ S^{k'} \left(z,u,\omega \right)$ means the third coordinate of $ S^{k'} \left(z,u,\omega \right)$.
\begin{proof}
By Lemmas \ref{Lemma apply Sinai to S} and  \ref{Lemma generator} we may apply Theorem \ref{Theorem 6.1} to 
\begin{equation} \label{Eq sys}
(\mathbb{T}^2 \times \mathbb{T} \times (\Pi_2 D)^\mathbb{N}, \, S, \, \kappa) \text{ with the generator } \mathcal{C}.
\end{equation} 
Thus, for our small $\epsilon>0$  there exists $l(\epsilon)$ such that for all  $l\ge l(\epsilon)$, we have a disjoint family $\lbrace C_i \rbrace_{i=1} ^{N(l,\epsilon)}$ such that
$$ \kappa ( \bigcup_{i=1} ^{N(l,\epsilon)} C_i ) >1-\epsilon$$
and for every $1\leq i\leq N(l,\epsilon)$,
$$  C_i  \subseteq \mathbb{T}^2 \times \mathbb{T} \times (\Pi_2 D)^\mathbb{N}$$
and 
\begin{equation} \label{Eq property 1}
\mathcal{N}_{n^{-l}} ( \Pi_{1,3} C_i)<e^{l \epsilon}.
\end{equation}
Furthermore, for $\kappa$ almost every $(z,u,\omega)$,
\begin{equation} \label{Eq property 2}
\mathcal{L} ( \overline{\lbrace R_\theta ^k (u): k\in \mathbb{N}, \quad S^k(z,u,\omega)\in C_i \rbrace } )\geq 1- \epsilon.
\end{equation}
To indicate the dependence of $l(\epsilon)$ on $R_\xi$,  in the following we will write $l(\epsilon,\xi)$ for $l(\epsilon)$. 
Notice that the measurable dependence of $l(\epsilon,\xi)$  on $\xi$ arises from the fact that our system \eqref{Eq sys}, specifically the measure $\kappa$, depends measurably on $R_\xi$.

Now, fix a $\kappa$ typical $(z,u,\omega)$ and a measure $\mu$ satisfying \eqref{Eq 8} and \eqref{Eq 9} (such a measure exists by Lemma \ref{Lemma apply Sinai to S}). We have the following estimate, which is a consequence of Theorem \ref{Theorem CP properties} part (4):
\begin{Lemma} \label{Lemm estimate meas} 
There exists some $r_0=r_0(\epsilon)>0$ such that for every $r<r_0$, $y\in \mathbb{T}^2$,  $k\in \mathbb{N}$, and $l\in \mathbb{N}$ large enough, we have
\begin{equation*}
\mathcal{N}_{n^{-l}} \left( \supp( \mu^{\mathcal{A}_k ^u (z)}) \setminus B(y,r) \right)\geq n^{l\cdot (\dim \mu -{o_{\epsilon} (1)})}.
\end{equation*}
\end{Lemma}

We also have the following estimate:

\begin{Claim} \label{Prop estimate}
For every $1\leq i \leq N(l,\epsilon)$ there is a set $C_i ' \subseteq C_i$ such that:
\begin{enumerate}
\item $\diam \left( \Pi_{1,3} (C_i ') \right) \leq n^{-l}$,

\item $\mathcal{N}_{n^{-l}} \left( \lbrace R_\theta ^k (u): k\in \mathbb{N}, \quad S^k(z,u,\omega)\in C_i ' \rbrace \right)\geq n^{(1-{o_{\epsilon}(1)})\cdot l}$.
\end{enumerate}
\end{Claim}
\begin{proof}
This is a consequence of the  properties \eqref{Eq property 1} and \eqref{Eq property 2}  of  $\lbrace C_i \rbrace_{i=1} ^{N(l,\epsilon)}$, and the pigeon hole principle applied to the family of sets   $\lbrace C_i \cap D: \quad   D\in \mathcal{C}_l \rbrace.$
\end{proof}

Fix some $i$ and $C_i '$ as in Claim \ref{Prop estimate}. Define
$$ \mathcal{N} := \lbrace k\in \mathbb{N}: \quad S^k(z,u,\omega)\in C_i ' \rbrace. $$
Then, by combining Lemma \ref{Lemm estimate meas}, \eqref{Eq property 1}, and Claim \ref{Prop estimate}, we can prove the inequality \eqref{First bound}: Indeed, let $X = \Pi_1 (C_i ')$. Since $C_i ' \subseteq C_i$, we have by \eqref{Eq property 1} that 
$$ \mathcal{N}_{n^{-l}} (X) \leq e^{l\cdot \epsilon}.$$
Also, writing $\mathcal{F} = \lbrace R_\theta ^k (u):\quad k\in \mathcal{N} \rbrace$, for every $t\in \mathcal{F}$ there  there exists a line with slope $m^{t}$ intersecting $X$ that supports a measure that satisfies Lemma \ref{Lemm estimate meas}. This follows by our choice of $\mathcal{N}$ and since $\Pi_1 \circ S^k(z,u,\omega) \in \supp \left( \mu^{\mathcal{A}_k ^u (z)} \right)$. Finally, consider the set 
$$ K = \left( \bigcup_{k\in \mathcal{N}} \supp( \mu^{\mathcal{A}_k ^u (z)}) \right) - X.$$
Then for any $t\in \mathcal{F}$, we can find some line $\ell$ with slope $m^{t}$ that supports a measure on $K$ that satisfies Lemma \ref{Lemm estimate meas}, and passes through an $l$-th level $n$-adic cube containing the origin. From this and Claim \ref{Prop estimate} part (2), one sees that
$$ \mathcal{N}_{n^{-l}} (K) \geq n^{l\cdot ( 1+\dim \mu -{o_{\epsilon,l}(1)})}.$$
It is well known that for every bounded sets $A,B\subseteq \mathbb{R}^2$ there is a constant $C_1$ such that
$$ \mathcal{N}_{n^{-l}} (A+B) \leq C_1 \cdot \mathcal{N}_{n^{-l}} (A) \cdot \mathcal{N}_{n^{-l}} (B).$$
Thus, since $\mathcal{N}_{n^{-l}} (X) \leq e^{l\cdot \epsilon}$, by the definition of $K$ and the last two displayed equations, the inequality \eqref{First bound} is proved.

As for the inequality \eqref{Second bound}, by Claim \ref{Prop estimate} we have 
$$ \diam \left( \Pi_{3} (C_i ') \right) \leq n^{-l}.$$
Therefore, for any $k,k'\in \mathcal{N}$ we have that (recalling our metric on the symbolic space  \eqref{The metric on symbolic})
$$ d( \Pi_3 \circ S^k (z,u,\omega),\quad \Pi_3 \circ S^{k'} (z,u,\omega)) \leq  n^{- l}.$$
So, since we have \eqref{Eq 9} at our disposal, for every $k'\in \mathcal{N}$ we have
$$  \Pi_1 \left( \bigcup_{k\in \mathcal{N}} \supp( \mu^{\mathcal{A}_k ^u (z)}) \right) \subseteq  A(  \Pi_3 \circ S^{k'} (z,u,\omega))  ^{(m^{-l})} \subseteq  A(  \Pi_3 \circ S^{k'} (z,u,\omega))  ^{(n^{-l})}$$ 
where $B^{(n^{-l})}$ is the $n^{-l}$-neighbourhood of a set $B$. Notice that we have used that $n<m$. From this, the inequality \eqref{Second bound} readily follows.
\end{proof}

\begin{Remark} \label{Remark - first remark}
In the proof above it was also established that since the mapping $\xi\to R_\xi$ is measurable,    $l(\epsilon,\xi)$ is also a measurable function of $\xi$.
\end{Remark}

Next, we estimate the covering number of the $\Pi_2$ projection of the set $\bigcup_{k\in \mathcal{N}} \supp( \mu^{\mathcal{A}_k ^u (z)})$ from Proposition \ref{Proposition two bounds}. Recall that the measure $\rho$ was defined in \eqref{Eq for rho}, and that by Lemma \ref{Lema rho is intesity}  its image under the base $n$ coding map $\Xi(\rho)\in \mathcal{P}(\mathbb{T})$ is $T_n$ invariant and
$$ \Xi(\rho) = \int \Pi_2 (\nu)\, dR(\nu,z,t,\omega).$$

From now on, we denote $\tilde{\rho}:=\Xi(\rho)$.  Recall that the ergodic decomposition of $R$ is given by
$$R =\int R_{\xi'} \, d\tau(\xi').$$
It follows that for $\tau$ almost every $\xi'$,  the measure
$$ \tilde{\rho}_{\xi'} = \int \Pi_2 (\nu) d R_{\xi'} (\nu,z,t,\omega)$$
is $T_n$ invariant and ergodic. Thus,
$$ \tilde{\rho} = \int \tilde{\rho}_{\xi'} \, d \tau ({\xi'})$$
is the ergodic decomposition of $\tilde{\rho}$. 

Fix $\tilde{\rho}_\xi$ for the ergodic component $R_\xi$ (recall \eqref{Assumption positive dim}) we have been working with so far. Recall that $X^{(n^{-l})}$ denotes the $n^{-l}$ neighbourhood of a set $X$. 
\begin{Proposition} \label{Proposition third bound}
Let $(z,u,\omega)$, $\mu$, and $\mathcal{N}$ be as in Proposition \ref{Proposition two bounds}. Then, for our small $\epsilon>0$  and all large $l\ge l(\epsilon,\xi)$, there exists a subset $\mathcal{N}'  = \mathcal{N}' _\xi \subseteq \mathcal{N}$ and a set $A=A_{\xi,\epsilon} \subseteq \mathbb{T}$ such that  for every $k\in \mathcal{N}'$,
$$\Pi_2 \mu^{\mathcal{A}_k ^u (z)} \left( A^{(n^{-l})} \right) \geq 1 -{o_{\epsilon}(1)}$$
such that a modified version of inequality \eqref{First bound} holds with
\begin{equation} \label{First bound modi}
 \mathcal{N}_{n^{-l}} \left( \bigcup_{k\in \mathcal{N}'} \supp( \mu^{\mathcal{A}_k ^u (z)}|_{[0,1]\times  A^{(n^{-l})}  } )  \right)\geq n^{l\cdot (\dim \mu+1 -{o_{\epsilon}(1)})}
\end{equation}
and we also have for some global constant $C_2$,
\begin{equation} \label{Third bound}
 \mathcal{N}_{n^{-l}} \left( \Pi_2 \left( \bigcup_{k\in \mathcal{N}'} \supp( \mu^{\mathcal{A}_k ^u (z)}|_{[0,1]\times A^{(n^{-l})}} ) \right)  \right)\leq C_2\cdot n^{l\cdot ( \dim \tilde{\rho}_\xi + {o_{\epsilon}(1)})}.
\end{equation}
\end{Proposition}
\begin{proof}
By Theorem \ref{Theorem 9.1}, since $\tilde{\rho}_\xi$ is $T_n$ invariant and ergodic, it is exact dimensional. By Egorov's Theorem there exists a compact set $A=A_{\xi,\epsilon}$, with $\dim_B A=\dim_H A$,  that varies measurably in $\xi$, such that
$$ \dim_B A = \dim \tilde{\rho}_\xi , \quad \text{ and } \tilde{\rho}_\xi(A)=1-o_\epsilon(1).$$
Also, since we have \eqref{Eq 8} at our disposal,
\begin{equation*} 
 \frac{1}{N}\sum_{k=1} ^N \Pi_2 \mu^{A_k ^u (z)} \rightarrow \int  \Pi_2 \nu \, dR_\xi (\nu,z,t,\omega) = \tilde{\rho}_\xi.
\end{equation*}
Therefore, since for every $l$ the set $A^{(n^{-l})}$ is open, there is a set $\mathcal{N}''\subset \mathbb{N}$ such that the density of $\mathcal{N}''$ in $\mathbb{N}$ is at least $1-{o_{\epsilon}(1)}$, and for every $k\in \mathcal{N}''$ we have
\begin{equation*}
\Pi_2 \mu^{A_k ^u (z)} \left( A^{(n^{-l})} \right) \geq 1-{ o_{\epsilon}(1)}.
\end{equation*}

Now, define
$$ \mathcal{N}' = \mathcal{N}\cap \mathcal{N}''.$$
Since the density of $\mathcal{N}''$ in $\mathbb{N}$ is at least $1-{o_{\epsilon}(1)}$, the density of $\mathcal{N}'$ in $\mathcal{N}$ is also at least $1-{o_{\epsilon}(1)}$. Then we arrive at the inequality \eqref{Third bound} since
\begin{equation*}
 \mathcal{N}_{n^{-l}} \left( \Pi_2 \left( \bigcup_{k\in \mathcal{N}'} \supp( \mu^{\mathcal{A}_k ^u (z)}|_{[0,1]\times A^{(n^{-l})}} ) \right)  \right)\leq  \mathcal{N}_{n^{-l}} \left( A^{(n^{-l})} \right) \leq  C_2\cdot n^{l\cdot (\dim \tilde{\rho}_\xi+ {o_{\epsilon}(1)})}.
\end{equation*}
Notice that the large $l$ we choose here depends on our set $A=A_{\xi,\epsilon}$, so $l=l(\epsilon,\xi)$. 

Finally, we need to justify the modified version of \eqref{First bound} given by \eqref{First bound modi}. To see this, notice that the outcome of Claim \ref{Prop estimate} is unchanged when we move to $\mathcal{N}'$, since the density of $\mathcal{N}'$ in $\mathcal{N}$ is at least $1-{o_{\epsilon}(1)}$. Thus, in order to run the same argument as at the end of Proposition \ref{Proposition two bounds}, we need to study what happens in the setting of Lemma \ref{Lemm estimate meas} with our extra conditioning on $[0,1]\times (A)^{n^{-l}}$. 

To this end, for every $k\in \mathbb{N}$, by Proposition \ref{Lemma entropy}
\begin{equation*}
\log \mathcal{N}_{n^{-l}} \left(  \supp( \mu^{\mathcal{A}_k ^u (z)}|_{[0,1]\times A^{(n^{-l})}} \right) \geq H(   \mu^{\mathcal{A}_k ^u (z)}|_{[0,1]\times A^{(n^{-l})}}, \mathcal{D}_{n^{l}} )
\end{equation*}
and since 
$$ \mu^{\mathcal{A}_k ^u (z)} \left( [0,1]\times A^{(n^{-l})} \right)\geq  1-{o_{\epsilon}(1)},$$
we have  (by \cite[Lemma 7.3]{wu2016proof})
$$ H(   \mu^{\mathcal{A}_k ^u (z)}|_{[0,1]\times A^{(n^{-l})}}, \mathcal{D}_{n^{l}} ) \geq H(   \mu^{\mathcal{A}_k ^u (z)}, \mathcal{D}_{n^{l}} ) - l\cdot \log n \cdot {o_{\epsilon}(1)}.$$
Finally, by another application of Theorem \ref{Theorem CP properties} part (4) we have
$$ H(   \mu^{\mathcal{A}_k ^u (z)}, \mathcal{D}_{n^{l}} ) \geq l\cdot\log n \cdot  ( \dim \mu -{o_{\epsilon}(1)}).$$
Combining the last four equations shows that indeed an analogue of Lemma \ref{Lemm estimate meas} holds in this modified situation as well, and we complete the proof of \eqref{First bound modi} in the same manner as in Proposition \ref{Proposition two bounds}.
\end{proof}

\begin{Remark} \label{Remark - the set Theta}
Recall that we use the same  $\epsilon>0$ for every component $R_\xi$ with positive dimension. As we already noted in Remark \ref{Remark - first remark}, the number $l(\epsilon,\xi)$ in Proposition \ref{Proposition two bounds} depends measurably on $\xi$. Similarly, the dependence of $l(\epsilon,\xi)$ in Proposition \ref{Proposition third bound}  is also measurable in $\xi$. Note  that the error terms ${o_{\epsilon} (1)}$ appearing  in the inequalities \eqref{First bound} and \eqref{Second bound} of Proposition \ref{Proposition two bounds}, and \eqref{First bound modi} and \eqref{Third bound} of Proposition \ref{Proposition third bound} go to zero as $\epsilon\rightarrow 0$ in a manner dependent on both $\epsilon$ and $\xi$. 

Recall that $R =\int R_{\xi} \, d\tau(\xi)$ is the ergodic decomposition of $R$. Now, let $\Theta$ be the set of all ergodic components of $R$ that have positive dimension. Since $R$ has positive dimension, $\tau(\Theta)>0$.  By an application of Egorov's Theorem, we may produce a subset $\Psi=\Psi(\epsilon) \subseteq \Theta$ of ergodic components of $R$ such that:
\begin{itemize}
\item $\tau (\Psi) >(1-\epsilon)\tau(\Theta)$.

\item $l$ can be chosen uniformly in both Proposition \ref{Proposition two bounds} and Proposition \ref{Proposition third bound} for all $R_\xi$ when $\xi \in \Psi$.
\end{itemize}  
Thus, for  ergodic components $R_\xi$ with $\xi \in \Psi$, the error terms $o_{\epsilon} (1)$ appearing  in the inequalities \eqref{First bound} and \eqref{Second bound} of Proposition \ref{Proposition two bounds}, and \eqref{First bound modi} and \eqref{Third bound} of Proposition \ref{Proposition third bound} go to zero as $\epsilon\rightarrow 0$  in a manner dependent only on $\epsilon$ (and not on $\xi$).
\end{Remark}

\subsubsection{Proof of Theorem \ref{Key prop}}
\noindent{\textbf{Proof of Part (1)}}  Let $R_\xi$ be an ergodic component such that $\xi\in \Psi$ (recall Remark \ref{Remark - the set Theta}), and $(z,u,\omega)$, $\mu$, $\tilde{\rho}_\xi$,  $\mathcal{N}$, $\mathcal{N}'$ and $A$ be as in Proposition \ref{Proposition two bounds} and Proposition \ref{Proposition third bound}. By these Propositions, for an error term {$o_{\epsilon}(1)$} that is independent of $\xi$, for $k' =\min \mathcal{N}$:
\begin{eqnarray*}
n^{l(\dim \mu+1 -{o_{\epsilon}(1)})} & \leq &  \mathcal{N}_{n^{-l}} \left( \bigcup_{k\in \mathcal{N}'} \supp( \mu^{\mathcal{A}_k ^u (z)}|_{[0,1]\times A^{(n^{-l})}} )  \right) \\
& \leq & \mathcal{N}_{n^{-l}} \left( \Pi_1 \bigcup_{k\in \mathcal{N}'} \supp( \mu^{\mathcal{A}_k ^u (z)}|_{[0,1]\times A^{(n^{-l})}} )  \right) \\
& \times & \mathcal{N}_{n^{-l}} \left(  \Pi_2 \bigcup_{k\in \mathcal{N}'} \supp( \mu^{\mathcal{A}_k ^u (z)}|_{[0,1]\times A^{(n^{-l})}} )  \right) \\
&\leq &  \mathcal{N}_{n^{-l}} \left( \Pi_1 \left( \bigcup_{k\in \mathcal{N}} \supp( \mu^{\mathcal{A}_k ^u (z)}) \right)  \right) \cdot C_2\cdot n^{l\cdot ( \dim \tilde{\rho}_\xi + {o_{\epsilon}(1)})} \\
& \leq & C_1 \cdot  \mathcal{N}_{n^{-l}} \left( A(  \Pi_3 \circ S^{k'} (z,u,\omega))  \right) \cdot  C_2\cdot n^{l\cdot ( \dim \tilde{\rho}_\xi + {o_{\epsilon} (1)})}.
\end{eqnarray*}
Taking $\log$  and dividing by $ l \log n$ we arrive at
\begin{equation} \label{Eq part (1)}
\dim \mu +1 -{o_{\epsilon} (1)} \leq \frac{\log \mathcal{N}_{n^{-l}} \left( A(  \Pi_3 \circ S^{k'} (z,u,\omega))  \right)}{l \log n} + \dim \tilde{\rho}_\xi.
\end{equation}
We remark that in equation \eqref{Eq part (1)} and the following calculations, we can absorb the {$o_{\epsilon}(1)$} factors that we encounter into each other,  which is possible since they are all uniform as {$\epsilon$ goes to $0$}, in a manner dependent only on $\epsilon$. Also, notice that  $k'$ is a measurable function of $\xi$.

Next, applying Theorem \ref{Theorem 9.1} we obtain
\begin{equation} \label{Eq for positive dim ergodic comp}
\dim \mu +1 -{o_{\epsilon} (1)} \leq \frac{\log \mathcal{N}_{n^{-l}} \left( A(  \Pi_3\circ  S^{k'} (z,u,\omega))  \right)}{l \log n} + \frac{h(\tilde{\rho}_\xi, T_n)}{\log n} .
\end{equation}

Now, equation \eqref{Eq for positive dim ergodic comp} holds as long as we are working with an ergodic component such that $\xi \in \Psi$. Recall that $\Theta$ denotes the set of  ergodic components of $R$ that have positive dimension. So, by the definition of $\dim R$, since $\tau(\Psi)> \tau( \Theta) - \epsilon$ and  for $\xi\not \in \Theta$ we have $\dim \mu =0$ for $R_\xi$ almost every $\mu$, via \eqref{Eq for positive dim ergodic comp} we see that
\begin{eqnarray*}
\dim R &=& \int_\Theta  \int  \dim \mu \, dR_\xi (\mu,z,u,\omega) d\tau(\xi) \\
&\leq & \int_\Psi  \int  \dim \mu \, dR_\xi (\mu,z,u,\omega) d\tau(\xi)+\epsilon \\
&\leq & \int_\Psi  \int  \left( \frac{\log \mathcal{N}_{n^{-l}} \left( A(  \Pi_3\circ  S^{k'} (z,u,\omega))  \right)}{l \log n} + \frac{h(\tilde{\rho}_\xi, T_n)}{\log n} -1  \right)  dR_\xi (\mu,z,u,\omega) d\tau(\xi) \\
&+&o_{\epsilon} (1)\\
&\leq & \int_\Theta  \int  \left( \frac{\log \mathcal{N}_{n^{-l}} \left( A(  \Pi_3\circ  S^{k'} (z,u,\omega))  \right)}{l \log n} + \frac{h(\tilde{\rho}_\xi, T_n)}{\log n} \right)  dR_\xi (\mu,z,u,\omega) d\tau(\xi) \\
&-& \tau(\Theta) +o_{\epsilon} (1)\\
&\leq& \int_{\Theta} \int  \frac{\log \mathcal{N}_{n^{-l}} \left( A(  \Pi_3 \circ S^{k'} (z,u,\omega))  \right)}{l \log n}  dR_\xi (\mu,z,u,\omega) d\tau(\xi) + \int_{\Theta} \frac{h(\tilde{\rho}_\xi, T_n)}{\log n}  d \tau(\xi) \\
&-& \tau(\Theta) +{o_{\epsilon} (1)}.
\end{eqnarray*}
where we have used that $\tilde{\rho}_\xi$ is constant when integrated against $R_\xi$. Next, applying \eqref{dim Q},
$$\dim R = \dim Q\geq \gamma$$
We thus arrive at the inequality
\begin{eqnarray} \label{Eq main inequality}
\gamma &\leq& \int_{\Theta} \int  \frac{\log \mathcal{N}_{n^{-l}} \left( A(  \Pi_3 \circ S^{k'} (z,u,\omega))  \right)}{l \log n}  dR_\xi (\mu,z,u,\omega) d\tau(\xi) + \int_{\Theta} \frac{h(\tilde{\rho}_\xi, T_n)}{\log n}  d \tau(\xi) \\
&-& \tau(\Theta) +{o_{\epsilon} (1)}
\end{eqnarray}
This implies Part (1) of Theorem \ref{Key prop}: Indeed, taking {$\epsilon\rightarrow 0$}, using that for every $\xi$ the measure $\tilde{\rho}_\xi$ is supported on $\Pi_2 (F)$ and that $\frac{h(\tilde{\rho}_\xi, T_n)}{\log n} = \dim \tilde{\rho}$, we obtain
$$\gamma \leq \tau(\Theta)\cdot \left( \sup_{\omega \in (\Pi_2 D) ^\mathbb{N}} \overline{\dim}_B A(\omega) + \dim_H \Pi_2 (F) -1 \right) = \tau(\Theta)\cdot (\dim^* F -1)$$
where in the last inequality we made use of Mackay's formula \cite{mackay2011assouad} for $\dim^* F$.

\noindent{\textbf{Proof of Part (2)}} By the $T$ invariance of $R_\xi$, writing $k'=k'(\xi)$ as before and  recalling \eqref{Eq number of returns}
\begin{eqnarray*}
\int_{\Theta} \int \left( \frac{\log \mathcal{N}_{n^{-l}} \left( A( \omega)  \right)}{l \log n} \right) dR_\xi (\mu,z,u,\omega) d\tau(\xi) & =& \\
 \int_{\Theta} \int \left( \frac{\log \mathcal{N}_{n^{-l}} \left( A( \omega)  \right)}{l \log n} \right) dR_\xi T^{k'}(\mu,z,u,\omega) d\tau(\xi) &=& \\
  \int_{\Theta} \int \left( \frac{\log \mathcal{N}_{n^{-l}} \left( A(\sigma^{\mathcal{R}(k',u)} \omega)  \right)}{l \log n} \right) dR_\xi (\mu,z,u,\omega) d\tau(\xi) &=& \\
 \int_{\Theta} \int \left( \frac{\log \mathcal{N}_{n^{-l}} \left( A(  \Pi_3 \circ S^{k'} (z,u,\omega))  \right)}{l\log n} \right) dR_\xi (\mu,z,u,\omega) d\tau(\xi).
\end{eqnarray*}
Combining this with \eqref{Eq main inequality} we obtain
\begin{equation*} 
\gamma  \leq  \int \int \left( \frac{\log \mathcal{N}_{n^{-l}} \left( A (\omega)  \right)}{l \log n} \right) dR_\xi (\mu,z,u,\omega) d\tau(\xi) + \int \frac{h(\tilde{\rho}_\xi, T_n)}{\log n} d \tau(\xi) - \tau(\Theta)+{o_{\epsilon} (1)}.
\end{equation*}
Notice that we have also removed the conditioning on the set $\Theta$ on the right hand side. Since $R_\xi$ is a disintegration of $R$, we  obtain
\begin{equation*} 
\gamma  \leq  \int \left( \frac{\log \mathcal{N}_{n^{-l}} \left( A (\omega)  \right)}{l \log n} \right) dR (\mu,z,u,\omega)  + \int \frac{h(\tilde{\rho}_\xi, T_n)}{\log n} d \tau(\xi)- \tau(\Theta)+{o_{\epsilon} (1)} .
\end{equation*}
Notice that up to an $o_l (1)$ factor,
$$ \frac{\log \mathcal{N}_{n^{-l}} \left( A (\omega)  \right)}{l \log n} = \frac{\log \mathcal{N}_{m^{-[l\cdot \theta]}} \left( A (\omega)  \right)}{ [l\cdot \theta] \log m}+o_l(1).$$
So, combining this with  Lemma \ref{Lemma 0.2}, using the affinity of entropy and letting $l\to\infty$, we get
\begin{equation*} 
\gamma  \leq  \frac{\sum_{j\in \Pi_2 (D)} \nu([j])\log a(j)}{\log m}  + \frac{h(\int \tilde{   \rho}_\xi d \tau(\xi), T_n)}{\log n} - \tau(\Theta)+{o_{\epsilon} (1)}.
\end{equation*}
Finally, $\int \tilde{   \rho}_\xi d \tau(\xi) = \tilde{\rho}$, and $(\tilde{\rho}, T_n)$ is a factor of $(\rho,\sigma)$, so we arrive at 
\begin{equation*} 
\gamma  \leq  \frac{\sum_{j\in \Pi_2 (D)} \nu([j])\log a(j)}{\log m}  + \frac{h(\rho, \sigma)}{\log n} - \tau(\Theta)+{o_{\epsilon} (1)} .
\end{equation*}
Taking $\epsilon\to 0$, this is Part (2) of Theorem \ref{Key prop}.

\noindent{\textbf{Proof of Part (3)}}   By Part (2) we have the following inequality:
\begin{equation} \label{Eq starting point}
\gamma+\tau(\Theta)  \leq \frac{\sum_{j\in \Pi_2 (D)} \nu([j])\log a(j)}{\log m} + \frac{h(\rho,\sigma)}{\log n}.
\end{equation}
Recall that by \eqref{Eq for nu}, \eqref{Eq for eta} and \eqref{Eq for rho}, the measures $\nu, \rho,\eta \in \mathcal{P}( (\Pi_2 D)^\mathbb{N})$ are $\sigma$ invariant and we have
$$ \rho = \theta\cdot \nu+(1-\theta)\cdot \eta.$$
We now show, via the equation above, that the right hand side of \eqref{Eq starting point} is bounded above by
$$ \dim_P F = \dim_B F= \frac{\log | \Pi_2 (D)|}{\log n} + \frac{\log \frac{|D|}{|\Pi_2 (D)|}}{\log m}. $$

To this end, by affinity of entropy, we have
$$h(\rho,\sigma) = h(\theta\cdot \nu+(1-\theta)\cdot \eta,\sigma) = \theta\cdot h( \nu,\sigma)+(1-\theta)\cdot h(\eta,\sigma).$$
Now, by the Kolmogorov-Sinai Theorem and Proposition \ref{Lemma entropy},
$$ h(\eta,\sigma) \leq H(\eta, \mathcal{D}) \leq \log |\Pi_2 (D)|,$$
where $\mathcal{D}$ is the first generation cylinder partition of $(\Pi_2 D)^\mathbb{N}$. By another application of the Kolmogorov-Sinai Theorem,
$$ h(\nu,\sigma) \leq H(\nu, \mathcal{D}) =\sum_{j\in \Pi_2 (D)} \nu([j])\cdot \log \frac{1}{\nu([j])}.$$
So, by the last two displayed inequalities and  recalling that $\theta=\frac{\log n}{\log m}$, we can bound
\begin{equation*}
\frac{\sum_{j\in \Pi_2 (D)} \nu([j])\log a(j)}{\log m} + \frac{h(\rho,\sigma)}{\log n} 
\end{equation*}
\begin{eqnarray*}
&=& \frac{\sum_{j\in \Pi_2 (D)} \nu([j])\log a(j)}{\log m} + \theta \cdot  \frac{h( \nu,\sigma)}{\log n}+(1-\theta) \cdot  \frac{h( \eta,\sigma)}{\log n}  \\
&\leq & \frac{\sum_{j\in \Pi_2 (D)} \nu([j])\log a(j)}{\log m} + \theta \cdot  \frac{\sum_{j\in \Pi_2 (D)} \nu([j])\cdot \log \frac{1}{\nu([j])}}{\log n}+(1-\theta)\cdot  \frac{\log |\Pi_2 (D)|}{\log n}  \\
&= & \frac{\sum_{j\in \Pi_2 (D)} \nu([j])\cdot \left( \log \left( \frac{a(j)}{\sum_j a(j)} \right) +\log  \frac{1}{\nu([j])} \right) +\log \left( \sum_j a(j) \right) }{\log m} +(1-\frac{\log n}{\log m}) \frac{\log |\Pi_2 (D)|}{\log n}\\
&\leq& \frac{\log \left( \sum_j a(j) \right) }{\log m} +(1-\frac{\log n}{\log m}) \frac{\log |\Pi_2 (D)|}{\log n}\\
&=& \frac{\log | \Pi_2 (D)|}{\log n} + \frac{\log \frac{|D|}{|\Pi_2 (D)|}}{\log m} = \dim_P F,
\end{eqnarray*}
where in the fourth inequality we used the Gibbs inequality (see Proposition \ref{Lemma entropy}).
Combining this with \eqref{Eq starting point} we see that
\begin{equation*} 
\gamma+\tau(\Theta)\leq \dim_P F,
\end{equation*}
thus  Part (3) of Theorem \ref{Key prop} is proved. \hfill{$\Box$}

\section{On the proof of Theorem \ref{Main Theorem} part (1)}  \label{Section proof part (1)}
\subsection{A Hausdorff dimension version of Theorem \ref{Key prop}}
The idea for the proof of Theorem \ref{Main Theorem} part (1) is similar to that of Theorem part (2), with some modifications.
Let $F$ be a Bedford-McMullen carpet with exponents $m>n$ and digits $D$, such that $m\not \sim n$. Write $\theta:=\frac{\log n}{\log m}$. Let $\ell_{0}$ be a line. We may assume, as in the proof of Theorem \ref{Main Theorem} part (2), that the slope of $\ell_0$ is $m^{u_0}\neq 0$ for $u_0 \in [0,1)$, and that our ambient space is $\mathbb{T}^2$ rather than $[0,1]^2$.

Let
$$ \gamma_1:= \dim_H \ell_0 \cap F$$
and fix some $\gamma <\gamma_1$. We will show that 
$$\gamma \leq \max \left\lbrace 0,\, \frac{\dim_H F}{\dim^* F} \cdot (\dim^* F-1) \right\rbrace . $$
It is clear that we may assume $\gamma>0$.

By Frostman's Lemma   we may find a probability measure $\mu_0 \in \mathcal{P}(\ell_0 \cap F)$ such that 
$$ \dim(\mu_0,z)\geq \gamma,\quad \text{ for } \mu_0 \text{ almost every } z.$$
In particular, $\mu_0$ is continuous (has no atoms). Fix a point $z_0 \in \ell_0 \cap F$ in the support of $\mu_0$ that satisfies the conclusion of Theorem \ref{Theorem CP exist} part (2).

Write 
 $$z_0=(x_0,y_0) = \left(\sum_{k=1} ^\infty \frac{x_k}{m^k}, \sum_{k=1} ^\infty \frac{y_k}{n^k}\right),\quad (x_k,y_k)\in D.$$
Notice that since $\mu_0$ is continuous, we may assume both $x_0,y_0 \notin \mathbb{Q}$, so that this representation is unique. Now, consider the sequence
$$ \omega_0=(y_1, y_2,...) \in (\Pi_2 D)^\mathbb{N} \subseteq \lbrace 0,...n-1 \rbrace^\mathbb{N}.$$ 
For every $k\in \mathbb{N}$, we define a sequence of measures on $(\Pi_2 D)^\mathbb{N}$:
 \begin{equation} \label{Eq for nu k}
\nu_k = \frac{1}{ [\theta^{-k+1}]} \sum_{k=1} ^{[\theta^{-k+1}]} \delta_{\sigma ^k (\omega_0)},
 \end{equation}
  \begin{equation} \label{Eq for eta k}
\eta_k = \frac{1}{[\theta^{-k}]- [\theta^{-k+1}]} \sum_{k=[\theta^{-k+1}]+1} ^{[\theta^{-k}]} \delta_{\sigma ^k (\omega_0)} .
 \end{equation}
 We shall require the following Claim to choose a subsequence of the scenery. Let $\mathcal{D}$ be the first generation partition of $(\Pi_2 D)^\mathbb{N}$. 
 \begin{Claim} \label{Claim choice of subsequnece}
 There exists a subseqeunce $N_j$ such that
 $$\limsup_{j\rightarrow \infty} \left( H(\eta_{N_j}, \mathcal{D}) - H(\nu_{N_j}, \mathcal{D}) \right) \leq 0.$$
 \end{Claim}
 \begin{proof}
 Suppose towards a contradiction that the Claim is not true. This means that for some $c>0$
$$  \liminf_{k\rightarrow \infty} \left( H(\nu_k, \mathcal{D}) - H(\eta_k, \mathcal{D}) \right) \leq  - c< 0.$$ 
So, for all large enough $k$ we have
 \begin{equation} \label{Eq contr}
 H(\eta_k, \mathcal{D})-H(\nu_k, \mathcal{D}) > \frac{c}{2}.
 \end{equation}
 
The crucial observation here is that
$$ \nu_{k+1} = \frac{[\theta^{-k+1}]}{[\theta^{-k}]} \nu_k + \frac{[\theta^{-k}]-[\theta^{-k+1}]}{[\theta^{-k}]} \cdot \eta_k.$$
So by concavity of entropy (Proposition \ref{Lemma entropy}) we have
$$H(\nu_{k+1}, \mathcal{D}) \geq  \frac{[\theta^{-k+1}]}{[\theta^{-k}]}\cdot H(\nu_k, \mathcal{D}) + \frac{[\theta^{-k}]-[\theta^{-k+1}]}{[\theta^{-k}]}\cdot H(\eta_k, \mathcal{D}).$$
Combining this with \eqref{Eq contr} we find that for all large enough $k$
$$H(\nu_{k+1}, \mathcal{D}) \geq   H(\nu_k, \mathcal{D}) + \frac{[\theta^{-k}]-[\theta^{-k+1}]}{[\theta^{-k}]}\cdot \frac{c}{2}.$$
The latter equation implies that $\lim_{k\rightarrow \infty} H(\nu_{k}, \mathcal{D})  =\infty$, which is a contradiction since for all $k$, 
$$ H(\nu_{k}, \mathcal{D}) \leq \log |\Pi_2 (D)|.$$
\end{proof}

From now on we work with the sequence $N_j$ from Claim \ref{Claim choice of subsequnece}. By Theorem \ref{Theorem CP exist} part (2), by perhaps passing to a further subsequence, there exists a distribution $Q$ such that
\begin{equation*} 
 \frac{1}{N_j} \sum_{k=0} ^{N_j-1} \delta_{M^k (\mu_0,z_0,u_0)} \rightarrow Q
 \end{equation*}
 where $Q$ is a CP distribution with 
\begin{equation*} 
\dim Q \geq \gamma .
\end{equation*}
Next, recalling \eqref{Eq for nu k} and \eqref{Eq for eta k}, by perhaps moving to yet a further subsequence, we assume that there are $\sigma$ invariant measures $\nu,\eta, \rho \in \mathcal{P}((\Pi_2 D)^\mathbb{N})\subseteq \mathcal{P}(\lbrace 0,...n-1 \rbrace^\mathbb{N})$ such that:
 \begin{equation} \label{Eq for nu new}
\nu_{N_j}\rightarrow \nu,
 \end{equation}
  \begin{equation} \label{Eq for eta new}
\eta_{N_j}\rightarrow \eta,
 \end{equation}
  \begin{equation} \label{Eq for rho new}
 \frac{1}{N_j} \sum_{k=1} ^{N_j} \delta_{\sigma ^k (\omega_0)} \rightarrow \rho.
 \end{equation}
It follows from \eqref{Eq for eta k} and \eqref{Eq for nu k} that  $\rho = \theta\cdot \nu + (1-\theta)\cdot \eta$. We also have, by Claim \ref{Claim choice of subsequnece}, the important inequality
\begin{equation} \label{Important inequality}
H(\eta, \mathcal{D})\leq H(\nu,\mathcal{D}).
\end{equation}
 
We can now formulate our required analogue of Theorem \ref{Key prop}:

\begin{theorem} \label{Key prop 2} Let $\lambda= Q( \lbrace \mu:\, \dim \mu >0 \rbrace)$. Then:
\begin{enumerate}
\item $\gamma \leq \lambda\cdot (\dim^* F -1)$.

\item
$$\gamma+\lambda \leq \frac{\sum_{j\in \Pi_2 (D)} \nu([j])\log a(j)}{\log m} + \frac{h(\rho,\sigma)}{\log n}.$$ 

\item $\gamma \leq \dim_H F -\lambda$.
\end{enumerate}
\end{theorem}

Theorem \ref{Key prop 2} implies Theorem \ref{Main Theorem} part (1), and this is completely analogues to the implication Theorem \ref{Key prop} $\Rightarrow$ Theorem \ref{Main Theorem} part (2).  The proof of parts (1) and (2) of Theorem \ref{Key prop 2} are  the same as the proof of the corresponding parts of Theorem \ref{Key prop} detailed in Section \ref{Section proof of key prop}. We thus omit the details. It remains to show that Part (2) implies Part (3), and this is the content of the next Section.
\subsection{Proof that part (2) implies part (3) in Theorem \ref{Key prop 2}}
Recall that $\gamma_1=\dim_H F\cap \ell_0$.  By Theorem \ref{Key prop 2} part (2) we have the following inequality:
\begin{equation} \label{Eq starting point 2}
\gamma+\lambda  \leq \frac{\sum_{j\in \Pi_2 (D)} \nu([j])\log a(j)}{\log m} + \frac{h(\rho,\sigma)}{\log n} 
\end{equation}
where  $\nu, \rho,\eta \in \mathcal{P}( \Pi_2 (D)^\mathbb{N})$ are $\sigma$ invariant and we have
$$ \rho = \theta\cdot \nu+(1-\theta)\cdot \eta.$$
We now show, via the equation above and \eqref{Important inequality}, that the right hand side of \eqref{Eq starting point 2} is bounded above by
$$ \dim_H F= \frac{\log \left( \sum_{j\in\Pi_2(D)} a(j)^\theta \right)}{\log n} ,$$
where we recall that $\theta=\frac{\log n}{\log m}$.

To this end, by affinity of entropy, the Kolmogorov-Sinai Theorem, and \eqref{Important inequality}
\begin{eqnarray*}
h(\rho,\sigma)& =& h(\theta\cdot \nu+(1-\theta)\cdot \eta,\sigma) \\
&=& \theta\cdot h( \nu,\sigma)+(1-\theta)\cdot h(\eta,\sigma) \\
&\leq & \theta\cdot H( \nu,\mathcal{D})+(1-\theta)\cdot H(\eta,\mathcal{D}) \\
&\leq & \theta\cdot H( \nu,\mathcal{D})+(1-\theta)\cdot H(\nu,\mathcal{D}) \\
&= &  H(\nu,\mathcal{D}).
\end{eqnarray*}

So,  we can bound
\begin{equation*}
\frac{\sum_{j\in \Pi_2 (D)} \nu([j])\log a(j)}{\log m} + \frac{h(\rho,\sigma)}{\log n} 
\end{equation*}
\begin{eqnarray*}
&\leq & \frac{\sum_{j\in \Pi_2 (D)} \nu([j])\log a(j)}{\log m} + \frac{H(\nu,\mathcal{D}) }{\log n}  \\
&= & \frac{\sum_{j\in \Pi_2 (D)} \nu([j])\log a(j)}{\log m} +  \frac{\sum_{j\in \Pi_2 (D)} \nu([j])\cdot \log \frac{1}{\nu([j])}}{\log n} \\
&=& \frac{1}{\log n} \left( \sum_{j\in \Pi_2 (D)} \nu([j]) \left( \log  \frac{1}{\nu([j])} + \log \left( \frac{a(j)^\theta}{\sum_{j\in \Pi_2 (D)} a(j)^\theta }  \right)\right) + \log \left( \sum_{j\in \Pi_2 (D)} a(j)^\theta \right)  \right)\\
&\leq & \frac{\log \left( \sum_{j\in \Pi_2 (D)} a(j)^\theta \right)}{\log n} = \dim_H F
\end{eqnarray*}
where in the last inequality we used the Gibbs inequality (see Proposition \ref{Lemma entropy}).
Combining this with \eqref{Eq starting point} we see that
\begin{equation*} 
\gamma_1+\lambda \leq \dim_H F.
\end{equation*}
\hfill{$\Box$}

\bibliography{bib}{}
\bibliographystyle{plain}

\end{document}